\documentclass[12pt]{article}
\usepackage{a4,fullpage,epsf,psfrag}
\usepackage{enumerate}
\usepackage{epsfig}
\usepackage{graphicx}
\usepackage{graphics}
\usepackage{latexsym}
\usepackage{epsfig}
\usepackage{amsfonts}
\usepackage{amsthm}
\usepackage{amsmath}
\usepackage{amssymb}
\usepackage[all]{xy}

\usepackage{amscd}
\usepackage{mathrsfs}
\usepackage{stmaryrd}
\usepackage{amsopn}
\usepackage[active]{srcltx}
\usepackage{amscd}
\usepackage{mathrsfs} 
\usepackage{amsopn} 
\usepackage[ps2pdf,breaklinks=true,colorlinks=true]{hyperref}
\usepackage[colorlinks=true]{hyperref}

\DeclareMathOperator{\JointSpec}{JointSpec}

\newtheorem{theorem}{Theorem}[section]
\newtheorem{prop}[theorem]{Proposition}

\newtheorem{lemma}[theorem]{Lemma}

\newtheorem{cor}[theorem]{Corollary}

\newcommand{\R}{{\mathbb R}}
\newcommand{\C}{{\mathbb C}}
\newcommand{\Z}{{\mathbb Z}}

\newcommand{\N}{{\mathbb N}}
\newcommand{\T}{{\mathbb T}}
\newcommand{\norm}[1]{\| #1\|}

\newcommand{\phy}{\varphi}

\newcommand{\op}[1]{\!\!\mathop{\rm ~#1}\nolimits}
\newcommand{\DD}{\mathrm{d}}

\newcommand{\scriptop}[1]{\!\!\mathop{\mbox{\rm \scriptsize ~#1}}\nolimits}

\newcommand{\ii}{\mathrm{i}}

\newcommand{\la}{\lambda} 
\newcommand{\al}{\alpha} 
 
\newcommand{\Om}{\Omega} 
 
\newcommand{\ga}{\gamma}

\newcommand{\om}{\omega} 
\newcommand{\Si}{\Sigma}

\newcommand{\De}{\Delta}
\newcommand{\si}{\sigma}

\newcommand{\hb}{\hbar}

\newcommand{\Hilbert}{{\mathcal{H}}}  

\newcommand{\Toeplitz}{{\mathcal{T}}}

\newcommand{\contravariant}{{\operatorname{cont}}}

\newenvironment{remark}{\refstepcounter{theorem}\par\medskip\noindent{\bf
Remark~\thetheorem~~}}{\unskip\nobreak\hfill\hbox{ $\oslash$}\par\bigskip}

\newcommand{\got}[1]{\mathfrak{#1}}

\newcommand{\pscal}[2]{\langle #1,#2\rangle}
\newcommand{\abs}[1]{\left|#1\right|}

\newcommand{\Cinf}{{\rm C}^{\infty}}

\newcommand{\RM}{\mathbb{R}}
\newcommand{\ZM}{\mathbb{Z}}

\newcommand{\NM}{\mathbb{N}}
\newcommand{\CM}{\mathbb{C}}

\renewcommand{\O}{\mathcal{O}}

\renewcommand{\geq}{\geqslant}
\renewcommand{\leq}{\leqslant}

\newcommand{\Hilb}{{\mathcal{H}}}

\begin{document}

\title{{\bf Isospectrality \\for quantum toric integrable systems}}

\author{Laurent Charles, {\'A}lvaro Pelayo, and San V\~{u} Ng\d{o}c}

\date{}

\maketitle

\begin{abstract}
We settle affirmatively the isospectral problem for quantum toric integrable systems: the
  semiclassical joint spectrum of such a system, given by a
  sequence of commuting Toeplitz operators on a sequence of Hilbert
  spaces, determines the classical integrable system given by the
  symplectic manifold and Poisson commuting functions, up to
  symplectomorphisms. We also give a full description of the semiclassical spectral theory of
quantum toric integrable systems.    This type of problem belongs to the
  realm of classical questions in spectral theory  going back
  to pioneer works of Colin de Verdi{\`e}re, Guillemin, Sternberg and others
  in the 1970s and 1980s.  
\end{abstract}

\section{Introduction} \label{sec:intro}
  
This paper gives a full description of the semiclassical spectral theory of
\emph{quantum toric integrable systems} in any finite dimension.  The classical limit
corresponding to quantum toric integrable systems are the so called symplectic
toric manifolds or \emph{toric systems}.  Such a system consists of a
compact symplectic $2n$\--manifold equipped with $n$ commuting
Hamiltonians $f_1,\, \ldots,\, f_n$ with periodic flows.  
The paper combines geometric techniques from the theory of toric manifolds, in
the complex-algebraic and symplectic settings, with recently developed
microlocal analytic methods for Toeplitz operators.

As a consequence of the spectral theory we develop, we answer 
the \emph{isospectrality} question for quantum toric integrable systems,
in any finite dimension: the semiclassical joint spectrum of a quantum toric
integrable system, given by a sequence of commuting Toeplitz operators
acting on quantum Hilbert spaces, determines the classical system
given by the symplectic manifold and Poisson commuting functions, up
to symplectic isomorphisms.  This type of symplectic isospectral problem belongs to the
 realm of classical questions in inverse spectral theory and microlocal analysis, 
 going back  to pioneer works of Colin de Verdi{\`e}re 
 \cite{CdV, CdV2}  and Guillemin\--Sternberg \cite{GuSt}  in the 1970s and 1980s.   
 
The question of isospectrality  in Riemannian geometry may be traced back to Weyl \cite{We1911,We1912} and
is most well known thanks to Kac's article \cite{Ka66}, who himself 
attributes the question to Bochner. Kac popularized 
the sentence: \emph{``can one hear the shape of a drum?"}, to refer to
this type of isospectral problem. The spectral theory developed in 
this paper exemplifies a striking difference with Riemannian
geometry, where this type of isospectrality rarely holds true, and suggests that
\emph{symplectic} invariants are much better encoded in spectral theory than 
\emph{Riemannian} invariants.  An approach to this problem for general
integrable systems is suggested in the last two authors' article \cite{PeVN2013}.
We refer to Section \ref{sec:remarks} for further remarks, and references, 
in these directions.

\subsection*{Microlocal analysis of integrable systems}

The notion of a quantum integrable system, as a maximal set of
commuting quantum observables, dates back to the early quantum
mechanics, to the works of Bohr, Sommerfeld and
Einstein~\cite{E1917}. However, the most basic results in the
symplectic theory of classical integrable systems like Darboux's
theorem or action-angle variables could not be used in
Schr{\"o}dinger's quantum setting at that time because they make use of the
analysis of differential (or pseudodifferential) operators in phase
space, known now as microlocal analysis, which was developed only in
the 1960s. The microlocal analysis of action-angle variables starts with the
works of Duistermaat~\cite{Du1974} and Colin de Verdi{\`e}re
\cite{CdV,CdV2}, followed by the semiclassical theory by Charbonnel
\cite{[5]}, and more recently by V\~{u} Ng\d{o}c \cite{san-spectral},
Zelditch and Toth \cite{[48],[49],[50]}, Charbonnel and Popov
\cite{ChPo1999}, Melin\--Sj{\"o}strand \cite{MeSj2003}, and many
others.
 
Effective models in quantum mechanics often require a compact phase
space, and thus cannot be treated using pseudodifferential
calculus. For instance the natural classical limit of a quantum spin
is a symplectic sphere.  The study of quantum action-angle variables
in the case of compact symplectic manifolds treated in this paper was
started by Charles \cite{Ch2003a}, using the theory of Toeplitz
operators.  In the present paper, we present global spectral results
for toric integrable systems; we use in a fundamental way Delzant's
theorem on symplectic toric manifolds \cite{d}.

Toric integrable systems \footnote{Toric integrable systems always
  have singularities of elliptic and transversally elliptic type, but
  do not have singularities of hyperbolic or focus-focus type. The
  local and semiglobal theory for regular points, elliptic and
  transversally elliptic singularities is now well understood both at
  the symplectic level (action-angle theorem of Liouville-Arnold-
  Mineur, Eliasson's linearization theorems), as well as the quantum
  level, see Charles \cite{Ch2003a} and V\~{u} Ng\d{o}c
  \cite{san-mono}.  This gives the foundation of the modern theory of
  integrable systems, in the spirit of Duistermaat's article
  \cite{Du1980}, but also of KAM-type perturbation theorems (see de la
  Llave's survey article \cite{llave}).}  have played an
influential role in symplectic and complex algebraic geometry,
representation theory, and spectral theory since T. Delzant classified
them in terms of combinatorial information (actually, in terms of a
convex polytope, see Theorem \ref{theo:delzant}).  A comparative study
of symplectic toric manifolds from the symplectic and complex
algebraic view points was given by Duistermaat and
Pelayo~\cite{DuPe2009}. A beautiful treatment of the classical theory
of toric systems is given in Guillemin's book \cite{Gu1994}.

Toeplitz operators are a natural generalization of Toeplitz matrices
(which correspond to Toeplitz operators on the unit disk). On $\RM^n$,
Toeplitz operators correspond to pseudo\-differential operators via
the Bargmann transform. Of course, such a correspondence cannot hold
in the case of a compact phase space, but it turns out that Toeplitz
operators always give rise to a semiclassical algebra of operators
with a symbolic calculus and microlocalization properties, which is
microlocally isomorphic to the algebra of pseudodifferential
operators. See the book by Boutet de Monvel and Guillemin for an
introduction to the spectral theory of Toeplitz operators
\cite{BoGu1981}.

\vskip 1em

\subsection*{Joint Spectrum}

In order to state our results, let us introduce the required
terminology. If $(M,\omega)$ is a symplectic manifold, a smooth map
$\mu=(\mu_1,\dots,\mu_n):M\to\RM^n$ is called a \emph{momentum map} for a
Hamiltonian $n$-torus action if the Hamiltonian flows $t_j\mapsto
\phy_{\mu_j}^{t_j}$ are periodic of period 1, and pairwise commute~:
$\phy_{\mu_j}^{t_j}\circ \phy_{\mu_i}^{t_i} = \phy_{\mu_i}^{t_i} \circ
\phy_{\mu_j}^{t_j}$, so that they define an action of $\R^n/\Z^n$. If
this action is effective and $M$ is compact.  $2n$\--dimensional and
connected, we call $(M, \om , \mu)$ a {\em symplectic toric manifold}.

By the Atiyah and Guillemin-Sternberg theorem, for any torus
Hamiltonian action on a connected compact manifold, the image of the
momentum map is a rational convex polytope~\cite{At1982,GuSt1982}. For
a symplectic toric manifold, the momentum polytope $\Delta \subset
\R^n $ has the additional property that for each vertex $v$ of
$\Delta$, the primitive normal vectors to the facets meeting at $v$
form a basis of the integral lattice $\Z^n$. We call such a polytope a
{\em Delzant polytope}.

A now standard procedure introduced by B. Kostant \cite{Ko1965, Ko1970,Ko1974,Ko1975} 
and  J.M. Souriau \cite{So1966, So1970} to quantize a symplectic compact manifold 
$(M,\om)$ is to introduce a prequantum bundle $\mathcal{L} \rightarrow M$,
that is a Hermitian line bundle with curvature $\frac{1}{i} \om$ and a
complex structure $j$ compatible with $\om$. One then defines the
quantum space as the space
$\mathcal{H}_k:=\mathrm{H}^0(M,\mathcal{L}^k)$ of holomorphic sections
of $\mathcal{L}^k$. The parameter $k$ is a positive integer, the
semi-classical limit corresponds to the large $k$ limit. A description of this
procedure, which is called \emph{geometric quantization}, is given by
Kostant and Pelayo in \cite{KoPe2012} from the angle of Lie theory and
representation theory.

Not all symplectic manifold have a complex structure or a prequantum
bundle. However a symplectic toric manifold always admits a compatible
complex structure, which is not unique. Furthermore a symplectic toric
manifold $M$ with momentum map $\mu:M\to\RM^n$ is prequantizable if
and only if there exists $c\in\RM^n$ such that the vertices of the
polytope $\mu(M)+c$ belong to $2 \pi \Z^n$ (see Section
\ref{sec:prequantization}). If it is the case, the prequantum bundle
is unique up to isomorphism.

In many papers, a prequantum bundle is defined as a line bundle with
curvature $\frac{1}{2 \pi i } \om$.  With this normalization, the
cohomology class of $\om$ is integral and the prequantization
condition for toric manifolds is that, up to translation, the momentum
polytope has integral vertices. This normalization may look simpler
than ours, which includes a $2\pi$\--factor. Nevertheless, our choice
is justified by the Weyl law. Indeed, with our normalization, the
dimension of the quantum space ${\mathcal{H}_k}$ is $$\Big(\frac{k}{2 \pi
}\Big)^n \operatorname{vol} ( M, \om) + \mathcal{O}(k^{n-1}).$$

Associated to such a quantization there is an algebra $\mathscr{T}
(M,{\mathcal{L}} , j)$ of operators
$$T=(T_k:\mathcal{H}_k\to\mathcal{H}_k)_{k\in \mathbb{N}^*}$$ called
\emph{Toeplitz operators}. This algebra plays the same role as the algebra of
semiclassical pseudodifferential operators for a cotangent phase
space. Here the semiclassical parameter is $\hbar=1/k$. A Toeplitz
operator has a \emph{principal symbol}, which is a smooth function on
the phase space $M$. If $T$ and $S$ are Toeplitz operators, then
$(T_k+k^{-1}S_k)_{k\in\NM^*}$ is a Toeplitz operator with the same
principal symbol as $T$.  If $T_k$ is Hermitian (i.e. self-adjoint)
for $k$ sufficiently large, then the principal symbol of $T$ is
real-valued.  Two Toeplitz operators $(T_k)_{k \in \N^*}$ and
$(S_k)_{k \in \N^*}$ \emph{commute} it $T_k$ and $S_k$ commute for
every $k$.

 \begin{figure}[h]
   \centering
   \includegraphics[height=2.5cm]{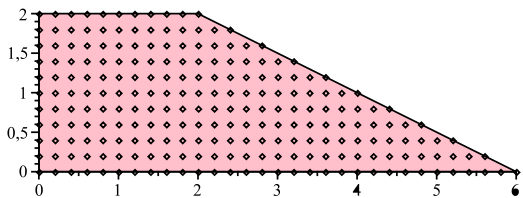}
   \caption{``Model Image" of the spectrum of a normalized quantum
   toric integrable system. }
   \label{fig:regular}
 \end{figure}

We shall also need the following definitions. If $P_1,\ldots,P_n$ are
mutually commuting endomorphism of a finite dimensional vector space,
then the \emph{joint spectrum of} $P_1,\ldots,P_n$ is the set of
$(\lambda_1,\dots,\lambda_n)\in\CM^n$ such that there exists a
non-zero vector $v$ for which $P_j v = \lambda_j v$, for all
$j=1,\dots,n$. It is denoted by $\op{JointSpec}(P_1,\, \ldots, \,
P_n)$. The \emph{Hausdorff distance} between two subsets $A$ and $B$
of $\R^n$ is
 $$
 {\rm d}_H(A,\,B):= \inf\{\epsilon > 0\,\, | \,\ A \subseteq
 B_\epsilon \ \mbox{and}\ B \subseteq A_\epsilon\},
$$
where for any subset $X$ of $\R^n$, the set $X_{\epsilon}$ is
$X_\epsilon := \bigcup_{x \in X} \{m \in \R^n\, \, | \,\, \|x - m \|
\leq \epsilon\}$.  If $(A_k)_{k \in \N^*}$ and $(B_k)_{k \in \N^*}$
are sequences of subsets of $\mathbb{R}^n$, we say that $$A_k = B_k +
\mathcal{O}(k^{-\infty}) $$ if ${\rm
  d}_H(A_k,\,B_k)=\mathcal{O}(k^{-N})$ for all $N \in \N^*$.

Our main result describes in full the joint spectrum of a quantum
toric integrable system.

\begin{theorem}[Joint Spectral Theorem] \label{theo:spectral} Let $(M,
  \, \omega, \, \mu : M \rightarrow \R^n)$ be a symplectic toric
  manifold equipped with a prequantum bundle ${\mathcal{L}}$ and a
  compatible complex structure $j$.  Let $T_1,\dots, T_n$ be commuting
  Toeplitz operators of $\mathscr{T} (M,{\mathcal{L}} , j)$ whose
  principal symbols are the components of $\mu$. Then the joint
  spectrum of $T_1,\, \ldots, \, T_n$ satisfies
  $$
  \textup{JointSpec}(T_1,\dots, T_n)= g \bigl( \Delta\cap \biggl( v +
  \frac{2 \pi}{k} \mathbb{Z}^n \biggr) ;\, k \bigr) +
  \mathcal{O}(k^{-\infty})
  $$
  where $\Delta=\mu(M)$, $v$ is any vertex of $\Delta$ and
  $g(\cdot;k):\RM^n\to\RM^n$ admits a $\Cinf$-asymptotic expansion of
  the form
  \[
  g(\cdot;k) = \textup{Id}+k^{-1}g_1+k^{-2}g_2+\cdots
  \]
  where each $g_j:\RM^n\to\RM^n$ is smooth.
\end{theorem}

Thus the joint spectrum of a quantum toric system can obtained by
taking the $k^{-1}\ZM^n$ lattice points in a polytope $\Delta$ (as in
Figure~\ref{fig:regular}), and applying a small smooth deformation $g$
(as in Figure~\ref{fig:deformed}).

\subsection*{Isospectrality}

We present next the isospectral theorem for toric systems. An easy
consequence of the previous theorem is that the momentum polytope
$\Delta$ is the Hausdorff limit of the joint spectrum of the quantum
system, that is $\Delta$ consists of the $\la \in \R^n$ such that for
any neighborhood $U$ of $\la$, $U \cap
\operatorname{JointSpec}(T_{1,k},\dots, T_{n,k}) \neq \emptyset$ when
$k$ is sufficiently large.

 \begin{figure}[h]
   \centering
   \includegraphics[height=6.2cm]{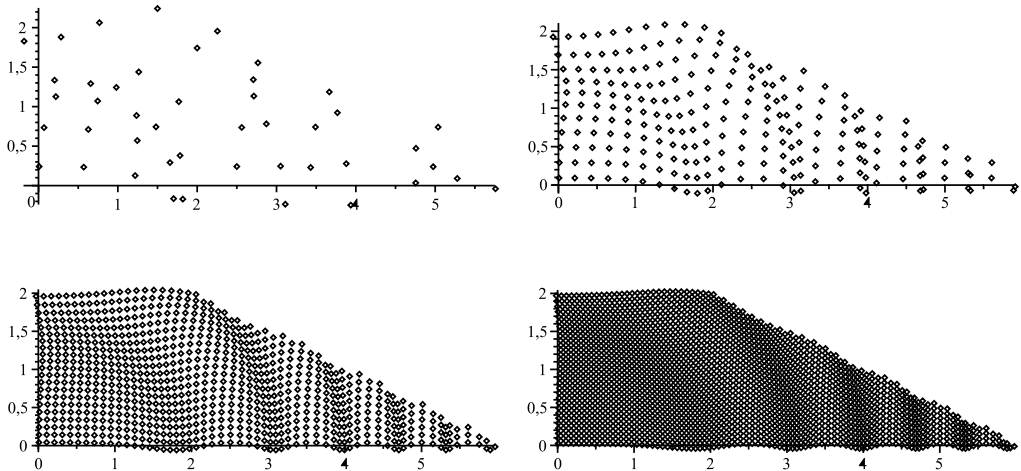}
   \caption{Sequence of images of the spectra of a quantum toric
   integrable systems as the semiclassical parameter $\hbar$ goes to
   $0$.  The spectra lie on a plane, so they corresponds to a
   four-dimensional integrable system with two degrees of freedom. In
   the Hausdorff limit corresponding to $\hbar=0$, the spectra
   converges to a polytope; this is proved for a general quantum
   toric system in any dimension by Theorem
   \ref{theo:inverse-spectral}. Therefore, one can recover the
   classical system from the semiclassical spectrum (i.e. the
   spectrum of the family of Toeplitz operators as $\hbar$ approaches
   $0$).}
   \label{fig:deformed}
 \end{figure}

Recall that two symplectic toric manifolds $(M, \om ,\mu)$ and $( M',
\om' , \mu')$ are \emph{isomorphic} if there exists a symplectomorphism
$\varphi: M \rightarrow M'$ such that $$\varphi^* \mu' = \mu.$$ By the
Delzant classification theorem \cite{d}, a symplectic toric manifold
is determined up to isomorphism by its momentum polytope. Furthemore,
for any Delzant polytope $\Delta$, Delzant constructed in \cite{d} a
symplectic toric manifold $(M_{\Delta} , \om_{\Delta} , \mu_{\Delta}
)$ with momentum polytope $\Delta$. Now we are ready to state our
isospectral theorem (see Figure \ref{fig:deformed} for an illustration 
of the semiclassical joint spectrum).

\begin{cor}[Isospectral Theorem] \label{theo:inverse-spectral}
  Let $(M, \, \omega, \, \mu : M \rightarrow \R^n)$ be a symplectic
  toric manifold equipped with a prequantum bundle ${\mathcal{L}}$ and
  a compatible complex structure $j$.  Let $T_1,\dots, T_n$ be
  commuting Toeplitz operators of $\mathscr{T} (M,{\mathcal{L}} , j)$
  whose principal symbols are the components of $\mu$. Then
$$ \De : = \lim_{k \rightarrow \infty}  \textup{JointSpec}(T_1,\dots, T_n) $$
is a Delzant polytope and $(M, \om, \mu)$ is isomorphic with
$(M_{\Delta}, \om_{\Delta}, \mu_{\Delta})$. In other words, one can
recover the classical system from the limit of the joint spectrum.
\end{cor}

This type of inverse conjecture is classical and belongs to the realm
of questions in inverse spectral theory, going back to similar
questions raised (and in many cases answered) by pioneer works of
Colin Verdi{\`e}re and Guillemin\--Sternberg in the 1970s and
1980s. Many other contributions followed their works, for instance
Datchev--Hezari--Ventura~\cite{DaHeVe} and
Iantchenko--Sj{\"o}strand--Zworski~\cite{IaSjZw2002}.  A few global
spectral results have also been obtained recently, for instance by
V\~{u} Ng\d{o}c~\cite{san-inverse} for one degree of freedom
pseudodifferential operators, or in the article by Dryden, Guillemin,
and Sena-Dias~\cite{dgs} in which an equivariant spectrum of the
Laplace operator is considered, and the references therein. See Section \ref{sec:remarks}
for further references.

\subsection*{Metaplectic correction} 
 
Introducing a metaplectic correction refers to twisting the prequantum bundle (or its powers) by a half-form bundle. The metaplectic correction allows to obtain an easier control of the subprincipal terms in the semiclassical limit. In the following theorem we improve the previous Joint spectral Theorem by giving the explicit description of the spectrum up to a $\mathcal{O}(k^{-2})$.

Recall that a half-form bundle of a complex manifold is a square root of its canonical bundle. Given a symplectic manifold $(M, \om)$ with a compatible complex structure, a prequantum bundle ${\mathcal{L}}$ and a half-form bundle $\delta$, the associated quantum space is $\Hilb_{\op{m},k} = {\rm H}^0 (M, {\mathcal{L}}^k \otimes \delta)$. We can define Toeplitz operators in this setting together with their principal symbols.  To state our result, we also need the notion of subprincipal symbol of a Toeplitz operator whose definition is recalled in Section \ref{sec:glob-quant-norm}. Two Toeplitz operators with the same principal symbol are equal up to $\mathcal{O}(k^{-2})$ if and only if they have the same subprincipal symbols.

As we will see in section \ref{sec:half-form}, a symplectic toric manifold with moment polytope $\Delta \subset \R^n$ has a half-form bundle if and only if there exists a vector $u \in \Z^n$ such that for any one-codimensional face $f$ of $\Delta$, the scalar product of a primitive normal of $f$ with $u$ is odd. Such a vector, if it exists, is uniquely determined modulo $(2 \Z)^n$. We denote it by $u_{\Delta}$.

\begin{theorem}[Joint Spectral Theorem with Metaplectic Correction] \label{theo:metapl-corr}
Let $(M,
  \, \omega, \, \mu : M \rightarrow \R^n)$ be a symplectic toric
  manifold equipped with a prequantum bundle ${\mathcal{L}}$, a
  compatible complex structure $j$ and a half-form bundle $\delta$.  Let $T_1,\dots, T_n$ be commuting
  Toeplitz operators of $\Hilb_{\op{m},k} $ whose
  principal symbols are the components of $\mu$. Then the joint
  spectrum of $T_1,\, \ldots, \, T_n$ satisfies
  $$
  \textup{JointSpec}(T_1,\dots, T_n)= g \bigl( \Delta\cap \biggl( v +
  \frac{2 \pi}{k} \bigl( \mathbb{Z}^n + u_{\Delta}/2 \bigr)  \biggr) ;\, k \bigr) +
  \mathcal{O}(k^{-\infty})
  $$
  where $\Delta=\mu(M)$, $v$ is any vertex of $\Delta$ and
  $g(\cdot;k):\RM^n\to\RM^n$ admits a $\Cinf$-asymptotic expansion of
  the form $$g(\cdot;k) = \textup{Id}+k^{-1}g_1+k^{-2}g_2+\cdots$$
  where each $g_j:\RM^n\to\RM^n$ is smooth. Furthermore $g_1$ is determined by 
$$ g_1^i ( E) = \int_0^1 f^i_1 (\varphi^t_{\mu^i} ( x))  \, {\rm d}t , \qquad \textup{for all}\, \, E \in \Delta,\,  x \in \mu^{-1}(E) $$ 
where $i=1, \ldots, n$, $f^i_1$ is the subprincipal symbol of $T^i$ and $\varphi^t_{\mu^i}$ is the Hamiltonian flow of $\mu^i$. 
\end{theorem}

Besides the average of the subprincipal symbols, it is interesting to note the shift by $ u_{\Delta}/2$ so that no eigenvalue lies on the boudary of $g ( \Delta)$ when $k$ is sufficiently large, cf. figure \ref{fig:deformed_metaplectic} for the spectrum of a model toric system with metaplectic correction.

\subsection*{Toeplitz quantization}

A natural question is how to decide whether a given integrable system
can be quantized. A discussion of this problem may be found in Garay
Van Straten~\cite{GaVa2010} and the references therein (they work with
pseudodifferential operators, instead of Toeplitz operators).
Concretely, given a prequantizable symplectic manifold endowed with an 
integrable system $(f_1,\dots,f_n)$, it may not be possible
to find a set of \emph{commuting} Toeplitz operators whose principal
symbols are $f_1,\dots,f_n$, respectively.  However, in the case toric
systems, we will obtain, as a byproduct of the proof of
Theorem~\ref{theo:spectral}, the following existence result.

\begin{theorem}[Existence of Toeplitz Quantization] 
  Let $(M, \, \omega, \, \mu \colon M \to \mathbb{R}^n )$ be a symplectic toric manifold
  equipped with a prequantum bundle ${\mathcal{L}}$ and a compatible
  complex structure $j$. Then there exists mutually commuting Toeplitz
  operators $T_1,\dots, T_n$ in $\mathscr{T} (M,{\mathcal{L}} , j)$
  whose principal symbols are the components of $\mu$.
\end{theorem}

The proofs in the paper combine geometric ideas from the theory of
toric manifolds in the complex and symplectic settings with microlocal
analytic methods dealing with semi-classical Toeplitz operators that
were developed by the first author \cite{Ch2003a, Ch2003b, Ch2006b,
  Ch2007}.

\section{Model for a symplectic toric
  manifold} \label{sec:model-sympl-toric}

We review the ingredients from the theory of symplectic toric
manifolds which we need for this paper, namely the Delzant
construction.

\begin{figure}[htbp]
  \begin{center}
    \includegraphics[width=6cm]{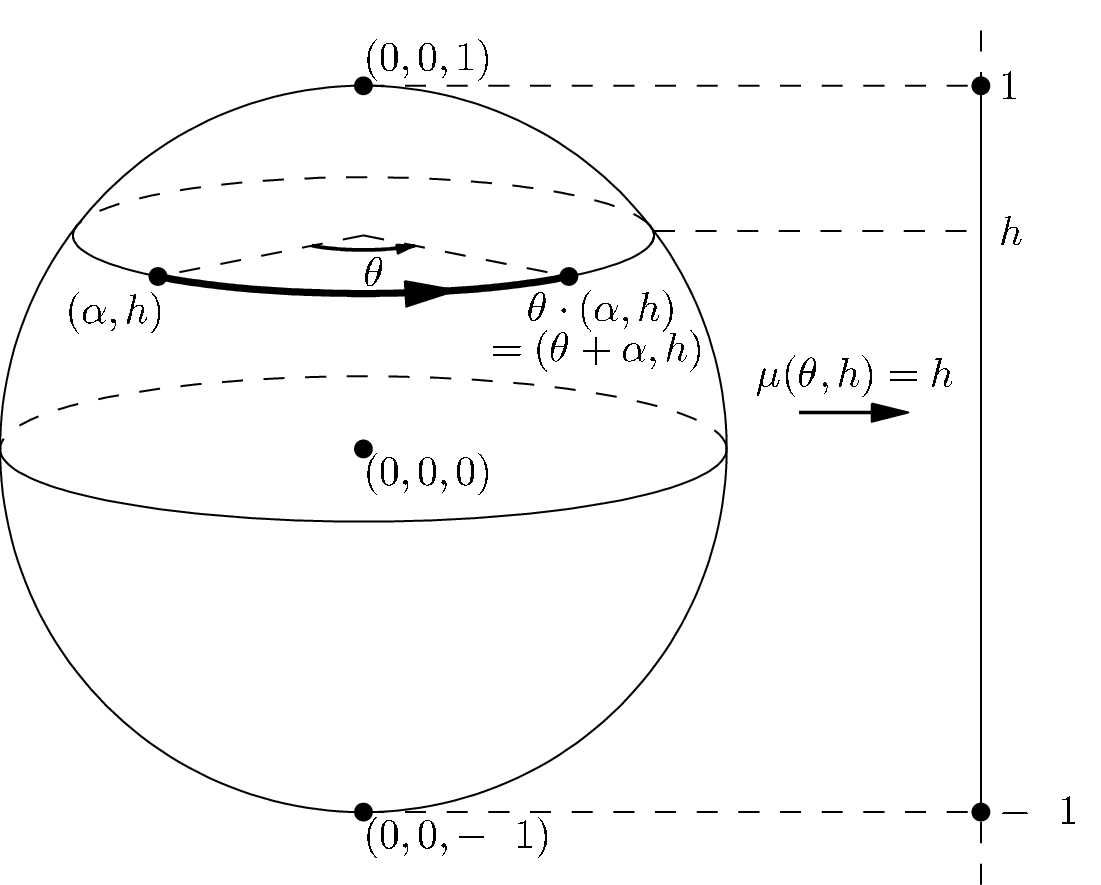}
    \caption{The momentum map for the $2$\--sphere $S^2$ is the height
      function $\mu(\theta,\, h)=h$. The image of $S^2$ under the
      momentum map $\mu$ is the closed interval $\Delta:=[-1,\,1]$. Note that
      as predicted by the Atiyah\--Guillemin\--Sternberg Theorem, 
      the interval $[-1,\,1]$ is equal to the image
      under $\mu$ of the set $\{(0,\,0,\,-1),\, (0,\, 0,\, 1)\}$ of
      fixed points of the Hamiltonian $S^1$\--action on $S^2$ by
      rotations about the vertical axis.}
    \label{fig1}
  \end{center}
\end{figure}

Let $T$ be an $n$\--dimensional torus. Denote by ${\got t}$ the Lie
algebra of $T$ and by ${\got t }_ \Z$ the kernel of the exponential
map $\op{exp}:{\got t}\to T$. We denote the isomorphism ${\got
  t}/{\got t}_{\Z} \to T$ also by $\op{exp}$. A \emph{symplectic toric manifold} 
  $(M, \om, T, \mu : M \rightarrow
{\got{t}}^*)$ is a symplectic compact connected manifold $(M, \om)$ of
dimension $2n$ with an effective Hamiltonian action of $T$ with
momentum $\mu$. When $T = \R^n / \Z^n$ so that $\got{t} \simeq \R^n$,
we recover the definition given in the introduction. Two symplectic
toric manifolds $(M, \, \omega, \, T, \, \mu )$ and $(M', \, \omega',\
T, \, \mu')$ are \emph{isomorphic} if there exists a symplectomorphism
$\varphi \colon M \to M'$ such that $\mu'\circ \varphi=\mu$. If it is
case, $\varphi$ intertwines the torus actions.

We present the construction of Delzant \cite{d} of symplectic toric
manifolds as reduced phase spaces.

\subsubsection*{Step 1 (\emph{Starting from a Delzant polytope $\Delta
    \subset \mathfrak{t}^*$})}

Let $\Delta$ be an $n$\--dimensional convex polytope in the dual Lie
algebra ${\got t}^*$.  We denote by $F$ and $V$ the set of all
codimension one faces and vertices of $\Delta$, respectively.  Every
face of $\Delta$ is compact. For every $v\in V$, we write $ F_v=\{
f\in F\mid v\in f\} .  $ The polytope $\Delta$ is called a {\em
  Delzant polytope} if it has the following properties, see Guillemin
\cite[p. 8]{Gu1994}.
\begin{itemize}
\item[i)] For each $f\in F$ there exists $X_f\in {\got t}_{\Z}$ and
  $\lambda _f\in\R$ such that the hyperplane which contains $f$ is
  equal to the set of all $\xi\in {\got t}^*$ such that $\langle
  X_f,\,\xi\rangle +\lambda _f=0$, and $\Delta$ is contained in the
  set of all $\xi\in {\got t}^*$ such that $\langle X_f,\,\xi\rangle
  +\lambda _f\geq 0$.

  \emph{Note}: The vector $X_f$ and constant $\lambda _f$ are made
  unique by requiring that they are not an integral multiple of
  another such vector and constant, respectively.
\item[ii)] For every vertex $v\in V$, the vectors $X_f$ with $f\in
  F_v$ form a $\Z$\--basis of the integral lattice ${\got t}_{\Z}$ in
  ${\got t}$.  \end{itemize} It follows that
$$
\Delta =\{\xi\in {\got t}^*\mid \langle X_f,\,\xi\rangle +\lambda
_f\geq 0 \quad\mbox{\rm for every}\quad f\in F\} .
$$
Also, $\# (F_v)=n$ for every $v\in V$.

\subsubsection*{Step 2 (\emph{The epimorphism  $\R ^F/\Z ^F \to T$ and the subtorus $N$})} 
Let $\pi \colon \R ^F \to {\got t}$ be defined by
$$\pi (t):=\sum_{f\in F}\, t_f\, X_f, \qquad t\in\R ^F.$$
Because, for any vertex $v$, the $X_f$ with $f\in F_v$ form a
$\Z$\--basis of ${\got t}_{\Z}$, we have $\pi (\Z ^F)={\got t}_{\Z}$
and $\pi (\R ^F)={\got t}$.  It follows that $\pi$ induces an
epimorphism $$\pi ' \colon \R ^F/\Z ^F=(\R /\Z)^F \to {\got t}/{\got
  t}_{\Z},$$ and we have the corresponding epimorphism
$\op{exp}\circ\pi ' \colon \R ^F/\Z ^F \to T$.

Write ${\got n}:=\op{ker}\pi$, a linear subspace of $\R ^F$,
and $$N:=\op{ker}(\op{exp}\circ\pi ') \subseteq \R ^F/\Z ^F,$$ a
compact commutative subgroup of the torus $\R ^F/\Z ^F$.  One can
check that $N$ is connected, and therefore isomorphic to ${\got
  n}/{\got n}_{\Z}$, where ${\got n}_{\Z}:={\got n}\cap\Z ^F$ is the
integral lattice in ${\got n}$ of the torus $N$.

\subsubsection*{Step 3 (\emph{Action of $N$ on $\mathbb{C}^F$})}

On the complex vector space $\C ^F$, we have the action of the torus
$\R ^F/\Z ^F$, where $t\in\R ^F/\Z ^F$ maps $z\in\C ^F$ to the element
$t\cdot z\in\C ^F$ defined by
$$
(t\cdot z)_f=\op{e}^{2\pi\scriptop{i}\, t_f}\, z_f,\quad f\in F.
$$
This action is Hamiltonian with momentum $\mu : \C^F \to (\R
^F)^*\simeq\R ^F$ given by
\begin{equation}
  \mu (z)_f=|z_f|^2/2 - \la_f = ( {x_f}^2 + {y_f}^2)/2 -\la_f , \quad f\in F.  
  \label{muf}
\end{equation}
Here $z_f=x_f+\op{i}y_f$, with $x_f,\, y_f\in\R$. Furthermore we work
with the symplectic form
\begin{gather} \label{eq:symp_form_C^F} \omega :=(\ii/4\pi)\,
  \sum_{f\in F}\,\op{d}\! z_f \wedge\op{d}\!\overline{z}_f =(1/2\pi
  )\,\sum_{f\in F}\,\op{d}\! x_f\wedge\op{d}\! y_f.
\end{gather}
The factor $1/2\pi$ is introduced in order to avoid an integral
lattice $(2\pi\,\Z )^F$ instead of our $\Z ^F$.

Hence $N$ acts on $\C ^F$ Hamiltonianly and the corresponding momentum
mapping is $ \mu _N:=\iota _{\got n}^*\circ\mu :\C ^F\to {\got n}^*, $
where $\iota _{\got n}:{\got n}\to\R ^F$ denotes the identity viewed
as a linear mapping from ${\got n}\subset\R ^F$ to $\R ^F$, and its
transposed $\iota _{\got n}^*:(\R ^F)^*\to {\got n}^*$ is the map
which assigns to each linear form on $\R ^F$ its restriction to ${\got
  n}$.

\subsubsection*{ Step 4 (\emph{The symplectic toric manifold
    $(M_{\Delta},\, \omega_{\Delta}, T, \mu_{\Delta})$}) }

It follows from Guillemin \cite[Theorems~1.6 and 1.4]{Gu1994} that $0$
is a regular value of $\mu _N$, hence the zero level set $Z$ of
$\mu_N$ is a smooth submanifold of $\C ^F$, and that the action of $N$
on $Z$ is proper and free.  The $N$\--orbit space $M_{\Delta}:=Z/N$ is
a smooth $2n$\--dimensional manifold such that the projection $p:Z\to
M_{\Delta}$ exhibits $Z$ as a principal $N$\--bundle over
$M_{\Delta}$.  Moreover, there is a unique symplectic form $\omega
_{\Delta}$ on $M_{\Delta}$ such that $p^*\omega _{\Delta}={\iota
  _Z}^*\omega$, where $\iota _Z$ is the identity viewed as a smooth
mapping from $Z$ to $\C ^F$.
 
On $(M_{\Delta}, \omega_{\Delta})$, the torus $(\R ^F/\Z ^F)/N\simeq
T$ acts effectively and Hamiltonianly, with momentum mapping $\mu
_\Delta:M_{\Delta} \to {\got t}^*$ determined by $\pi ^*\circ\mu
_\Delta\circ p =\mu |_Z$, and\footnote{See Guillemin \cite[Theorem
  1.7]{Gu1994}.}  $\mu _\Delta(M_{\Delta})=\Delta$.  In other words,
$(M_{\Delta}, \omega_{\Delta}, T, \mu_{\Delta})$ is a symplectic toric
manifold with momentum map image equal to $\Delta$.

\begin{theorem}[Delzant's Theorem] \label{theo:delzant} Any abstract
  symplectic toric manifold $(M,\, \omega,\, T, \, \mu)$ with momentum
  polytope $\Delta \subset \mathfrak{t}^*$ is isomorphic to
  $(M_{\Delta},\, \omega_{\Delta}, T,\, \mu_{\Delta})$. Moreover, two
  symplectic toric manifolds $(M_{\Delta}, \,\omega_{\Delta}, \, T, \,
  \mu_{\Delta})$ and $(M_{\Delta'}, \,\omega_{\Delta'},\, T,\,
  \mu_{\Delta'} )$ are isomorphic if and only if $\Delta=\Delta'$.
\end{theorem}

Since the action of $\R^F/ \Z^F$ preserves the complex structure of
$\C^F$, $M_\Delta$ inherits by reduction a complex structure
compatible with $\om_{\Delta} $ and invariant by the action of $T$
(cf. \cite[Theorem~3.5]{GuSt}).  So $M_\Delta$ is a K{\"a}hler
manifold.

\section{Prequantization} \label{sec:prequantization}

Recall that a {\em prequantum bundle} on a symplectic manifold $(M,
\om)$ is a Hermitian line bundle with base $M$ endowed with a
connection of curvature $\frac{1}{i} \om$. An \emph{automorphism of
  prequantum bundle} is an automorphism of Hermitian line bundles
preserving the connection.

Let $\mathcal{L}$ be a prequantum bundle over $(M, \om)$. Consider an
action of a Lie group $G$ on $\mathcal{L}$ by prequantum bundle
automorphisms. This action lifts an action of $G$ on $M$. One proves
that the latter action is Hamiltonian and has a natural momentum $\mu$
determined by the following condition: the induced action of the Lie
algebra $\mathfrak g$ on $\Cinf (M, \mathcal{L})$ is given by the
Kostant-Souriau operators
$$ f \rightarrow  \nabla_{X^\sharp}f  + {\rm i} \langle \mu , X \rangle f   , \qquad X \in \mathfrak{g}, $$ 
where we denote by $X^\sharp$ the infinitesimal action of $X$ on $M$
and by $\nabla$ the covariant derivative of the prequantum bundle.

If $G$ and $M$ are connected, the action on ${\mathcal{L}}$ is
conversely determined by the action on $M$ and the momentum
$\mu$. Furthermore if $G$ is a torus $T$ acting in a Hamiltonian way
on a connected $M$, then the momentum associated to a lift to
${\mathcal{L}}$, if it exists, is unique up to a translation by a
vector of $2 \pi \mathfrak{t}^*_{\Z}$.  Let us consider the case of
symplectic toric manifolds.  In the following result,  the \emph{length of an edge} $e$ 
of the polytope $\Delta \subset \mathfrak{t}^*$ 
is the smallest positive real $\ell$
such that $ e / \ell \in \mathfrak{t}^*_{\Z}$.

\begin{figure}[htbp]
  \begin{center}
    \includegraphics[width=7cm]{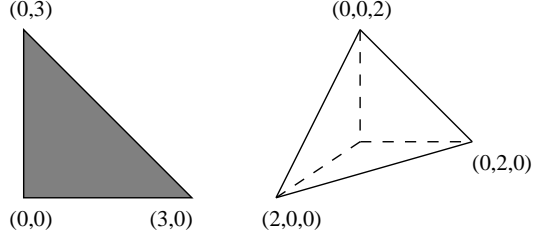}
    \caption{Delzant polytopes corresponding to the complex projective
      spaces $\mathbb{CP}^2$ and $\mathbb{CP}^3$ equipped with scalar
      multiples of the Fubini\--Study symplectic form.}
    \label{fig:AFF}
  \end{center}
\end{figure}

\begin{prop} \label{prop:prequantization} A symplectic toric manifold
  $(M,\omega,T,\mu)$ with momentum polytope $\Delta \subset \mathfrak{t}^*$ admits a
  prequantum bundle $\mathcal L$ if and only if the edges of $\Delta$
  belong to $2\pi \got{t}_\Z^*$.  If it is the case, the prequantum
  bundle is unique up to isomorphism and the action can be lifted to
  ${\mathcal {L}}$ in such a way that each element acts as a
  prequantum bundle automorphism. The moments corresponding to the
  possible lifts are the ones such that the vertices of the associated
  polytope belong to $2 \pi \got{t}^*_\Z$.
\end{prop}

\begin{proof} 
  Let us check that the condition is necessary. Let $(M, \om)$ be a
  symplectic toric manifold with momentum $\mu$. Let $e$ be an edge of
  $\Delta= \mu (M)$.  The length of $e$ is the smallest positive real
  $\ell$ such that $ e / \ell \in \mathfrak{t}^*_{\Z}$. Working with
  the local symplectic charts associated to the vertices of $e$, one
  checks that $\mu^{-1} (e)$ is a symplectic 2-sphere embedded in $M$
  with volume $ \ell$. Because the image by the map $${\rm H}^2 ( M,
  \Z) \to {\rm H}^2 ( M, \R)$$ of the Chern class of a prequantum
  bundle of $M$ is $\frac{1}{2\pi} [\om]$, if $M$ is prequantizable,
  then $\ell/2 \pi$ is necessarily integral, so $e \in 2 \pi
  \got{t}_\Z^*$.  To conclude, observe that $(2\pi)^{-1} \De$ can be
  translated to a polytope with integral vertices if and only if its
  edges have integral lengths. The uniqueness of the prequantum bundle
  follows from the fact that any symplectic toric manifold is simply
  connected.

  Conversely assume that $(M, \om,T,\mu)$ is a symplectic toric manifold
  with a momentum such that the lengths of the edges of the momentum
  polytope $\Delta$ are integral multiple of $2 \pi$. Modifying this
  momentum by an additive constant in $\got{t}^*$ if necessary, we may
  assume that the vertices of $\Delta $ belong to $2 \pi \got{t}^*_{\mathbb{Z}}$.
  Then the Delzant construction provides a prequantum bundle
  $\mathcal{L}$ and a lift of the action to ${\mathcal{L}}$. Since we
  need this construction in the sequel, let us give some
  details. Consider the prequantum bundle of $\C^{F}$ given by
  $\mathcal{L}_F := \C^F \times \C$ with connection
  \[ \nabla \,=\, d+\frac{1}{4 \pi \ii}\sum_{f\in F} (x_f\DD y_f
  -y_f\DD x_f)\,= \,d+ \frac{1}{8 \pi } \sum_{f\in F} (z_f \DD
  \bar{z}_f - \bar{z}_f\DD z_f). \] Here and in the remainder of the
  proof, we use the same notations as in section
  \ref{sec:model-sympl-toric}. In particular the symplectic form of
  $\C^F$ is given by (\ref{eq:symp_form_C^F}). Since the vertices of
  $\Delta $ belong to $2 \pi \got{t}^*_Z$, the $\la_{f}$'s defining
  the faces of $\Delta$ are integral multiple of $2\pi$. So we can
  lift the action of $\R^F / \Z^F$ on $\C^F$ to $\mathcal{L}_F$ by
$$  t \cdot (z,u) = ( t \cdot z , u \op{e}^{ {\rm i} \sum_{f \in F}  t_f \la_f} ). $$
With a straightforward computation, one checks that this action
preserve the prequantum bundle structure and that its associated
momentum is $\mu$ defined in (\ref{muf}).

Recall that the Delzant manifold $(M_\Delta, \om _{\Delta})$ is the
quotient of $Z = \mu _N^{-1} ( 0)$ by the action of subtorus $N$ of
$\R^F / \Z^F$. Then quotienting by $N$ the restriction of
$\mathcal{L}_F$ to $Z$, we obtain a prequantum bundle
$\mathcal{L}_\De$ over $M_\Delta$
(cf. \cite[Theorem~3.2]{GuSt}). Furthermore the group $T = (\R^F/
\Z^F) /N$ acts on $\mathcal{L}_\De$. This action preserves the
prequantum bundle structure, its associated momentum is the
application $\mu_\Delta$ defined in Theorem \ref{theo:delzant}.
\end{proof}

\section{Quantum model} \label{sec:quantum-model}

In this section we introduce a
quantum model for quantum toric system and compute its spectrum.

Consider a Delzant polytope $\Delta \subset \mathfrak{t}^*$ with
vertices in $2\pi \got{t}_{\Z}^*$. Then as explained in Section
\ref{sec:prequantization}, the Delzant manifold $( M_{\Delta} ,
\om_{\Delta}, T, \,\mu_{\Delta})$ admits a prequantum bundle ${\mathcal{L}}_{\Delta}$
unique up to isomorphism. This line bundle has a unique holomorphic
structure compatible with the complex structure of $M_{\De}$ provided
by reduction and with the connection.  So for any positive integer
$k$, we can define the quantum space
$$ \Hilb_k^{\Delta}: = {\rm H}^0 ( M_{\Delta} , {\mathcal{L}}^k_{\Delta} ) $$  
which consists of the holomorphic sections of
${\mathcal{L}}^k_{\Delta}$. 

For any $X \in \mathfrak{t}$, consider the
\emph{rescaled Kostant-Souriau operators}
\begin{eqnarray} \label{eq:Txmux} T_{X,k} := \langle \mu_\Delta , X
  \rangle + \frac{1}{{\rm i}k} \nabla_{X^{\sharp}}: \Hilb_k^{\Delta}
  \rightarrow \Hilb_k^{\Delta}
\end{eqnarray}
This operator is well-defined because the complex structure is
invariant by the action of $T$ on $M_{\Delta}$. The rescaling has the
effect that the $T_{X,k}$'s are self-adjoint and that their joint
spectrum is the intersection of the polytope $\Delta$ and a rescaled
lattice. The precise result is the following.

\begin{theorem} \label{theo:tt} There is an orthogonal decomposition
  of the quantum space   $\Hilb_k^{\Delta}$    into a direct sum of lines:
$$ \Hilb_k^{\Delta}  = \bigoplus_{\ell \in (\frac{2\pi}{k}  \mathfrak{t}_{\Z}^* ) \cap \Delta}  D_{\ell}^k$$
such that, for any $X \in {\mathfrak{t}}$,
$$ T_{X,k} \Psi = \ell (X) \Psi , \qquad \textup{for all}\,\,\, \Psi \in D_{\ell}^k.$$ 
\end{theorem}

\begin{proof}
  The proof uses the Delzant construction recalled in Section
  \ref{sec:model-sympl-toric} at the quantum level. Consider the same
  prequantum bundle $\mathcal{L}_F \rightarrow \C^F$ as in the proof
  of Proposition \ref{prop:prequantization}.  The associated quantum
  space $\mathcal{B}_k$ is the space of holomorphic sections $\psi$ of
  $\mathcal{L}^k_F$ such that
  \[
  \int_{\CM^F} \abs{\psi}^2(z) \nu(z) < \infty,
  \]
  where $\nu$ is the Liouville measure.  Here holomorphic means that
  the covariant derivative with respect to antiholomorphic vectors
  vanishes. Such a holomorphic section can be written
$$\psi=\op{e}^{-k\abs{z}^2/8 \pi}f,$$ 
where $f$ is a plain holomorphic function on $\CM^F$. Here
$\abs{z}^2=\sum_f \abs{z_f}^2$. So $\mathcal{B}_k$ can be identified
with the usual Bargmann space, that is the space of holomorphic
functions on a $\C^p$ whose square is integrable with respect to a
given gaussian weight. An orthogonal basis of $\mathcal{B}_k$ is given
by the family
\[
\psi_\alpha = \op{e}^{-\frac{k}{8 \pi}\abs{z}^2}z^\alpha, \qquad
\alpha\in \NM^F.
\]
Consider the Kostant-Souriau operator associated to momentum $\mu$
given in (\ref{muf})
\[
S_{X,k} = \pscal{\mu}{X} + \frac{1}{{\rm i}k} \nabla_{X^{\sharp}}.
\]
If $X=e_f$ a straightforward computation shows that
$$S_{X,k} \Bigl( \op{e}^{-\frac{k}{8 \pi}\abs{z}^2} g(z) \Bigr)=
\op{e}^{-\frac{k}{8 \pi }\abs{z}^2} \Biggl( \frac{2 \pi }{ k}
z_f\partial_{z_f}g - \lambda_f g \Biggr).$$ We deduce that for any $X
\in \RM^F$ and $\al \in \N^F$,
\begin{eqnarray} \label{eq:tx} S_{X,k}(\psi_\alpha) =\Bigl\langle X,\,
  \frac{2 \pi }{k} \alpha -\lambda \Bigr\rangle \psi_\alpha.
\end{eqnarray}
Recall that the Delzant space was defined as the symplectic quotient
of $\C^{F}$ by the subtorus $N$ of $\R^F / \Z^F$. The corresponding
space at the quantum level is
\[
\mathcal{B}^{\mathfrak{n}}_k := \{\psi\in \mathcal{B}_k \quad | \quad S_{X,k}
\psi=0 \quad \textup{for all}\,\, X\in \frak{n}\}.
\]
We call it the \emph{reduced quantum space}. We deduce from Equation
(\ref{eq:tx}) that a basis of $\mathcal{B}^{\mathfrak{n}}_k$ consists
of the $\Psi_{\al}$'s such that $\alpha \in \N^F$ satisfies $$\langle
X , \frac{2 \pi}{k}\al-\lambda \rangle =0$$ for all $X \in \mathfrak
n$.  Equivalently, $\frac{2\pi}{k} \alpha$ runs over $(\frac{2\pi}{k}
\mathbb{N}^F ) \cap (\lambda+\ker (\iota_{\mathfrak n}^*))$. This set
is in bijection with $(\frac{2 \pi}{k} \mathfrak{t}_{\mathbb{Z}}^* )
\cap \Delta$.
\begin{lemma} \label{lemm:bij} The map $\pi^* + \lambda$ from
  $\mathfrak{t}^*$ to $(\mathbb{R}^F)^*$ restricts to a bijection 
  $$(\frac{2 \pi}{k} \mathfrak{t}_{\mathbb{Z}}^*) \cap \Delta \longrightarrow (\frac{2 \pi}{k} \mathbb{N}^F ) \cap (\lambda+\ker
  (\iota_{\got{n}}^*)). 
  $$
  Furthermore, for any $X \in \mathbb{R}^F$ and
  $\ell \in \frac{2 \pi }{k} \mathfrak{t}_{\mathbb{Z}}^* \cap \Delta$
  we have that
 $$
 S_{X,k} \bigl( \psi_{\alpha}\bigr) =\langle \ell ,\, \pi(X) \rangle
 \, \psi_{\alpha},\qquad \text{if}\,\,\,\,\,\,\, \frac{2 \pi }{k} \al
 =\pi^*(\ell)+\lambda.
 $$
\end{lemma}

\begin{proof}[Proof of Lemma \ref{lemm:bij}]
  We know that $\pi^*$ is injective with image
  $\ker(\iota_{\mathfrak{n}}^*)$.  Using the Delzant condition on a
  vertex (Step 1 in Section \ref{sec:model-sympl-toric}), one sees
  that $\pi^*$ restricts to a bijection from $\mathfrak{t}_{\Z}^*$ to
  $(\Z^F)^* \cap \ker(\iota^*_{\mathfrak{n}})$. By the prequantization
  condition, $ \la \in 2 \pi (\Z^F)^* \subset \frac{2 \pi }{k} (
  \Z^F)^*$. So $\pi^* + \lambda$ restricts to a bijection from
  $\frac{2 \pi}{k} \mathfrak{t}_{\mathbb{Z}}^* $ to $(\frac{2 \pi}{k}
  \mathbb{Z}^F ) \cap (\lambda+\ker (\iota_{\got{n}}^*))$.
  Furthermore, the proof to show that $\mu_\Delta ( M_{\Delta}) =
  \Delta$ implies that
  \begin{eqnarray}
    \pi^*(\Delta)+\lambda=\mathbb{R}^F_{\ge 0} \cap (\lambda + \ker(\iota^*_{\got{n}})).
  \end{eqnarray}
  This implies the first part of the lemma.  The second assertion
  follows from (\ref{eq:tx}).
\end{proof}

The end of the proof of Theorem \ref{theo:tt} is an application of the
``quantization commutes with reduction'' theorem of Guillemin\--Sternberg. Lifting the action
of $\T^F = \RM^F/\ZM^F$ to $\mathcal{L}_F$ as in the proof of
Proposition \ref{prop:prequantization}, we get an action of $\T^F$ on
$\mathcal{B}_k$. The reduced quantum space is the subspace of
$\mathcal{B}_k$ of $N$-invariant vectors, in other words
$\mathcal{B}^{\mathfrak{n}}_{k}=(\mathcal{B}_{k})^N$. Now by
Guillemin-Sternberg theorem (\cite{GuSt}), we have an isomorphism
$$\Phi_k : \mathcal{B}^{\mathfrak{n}}_k \rightarrow \mathcal{H}_k^\Delta .$$
The proof in \cite{GuSt} given in the compact case extends to our
setting by \cite{Ch2006a} and \cite{Gu1994}.  Furthermore, under the
isomorphism $\Phi_k$, the action of the torus $T$ on
$\mathcal{H}_k^{\Delta}$ corresponds to the action of $\mathbb{T}^F/N$
on $\mathcal{B}^{\mathfrak{n}}_k$. Then, passing to the level of Lie
algebras, we get the following relation between the Kostant-Souriau
operators:
$$ \Phi_k \Bigl(  S_{X,k} \Psi \Bigr) = T_{\pi(X),k} \Phi_k ( \Psi), \qquad \textup{for all}\,\,  \Psi \in 
\mathcal{B}^{\mathfrak{n}}_k  .$$ 
This concludes the proof of Theorem \ref{theo:tt}.
\end{proof}

From the previous theorem, we deduce the following quantum normal
form.  Consider a Delzant polytope $\De$ in the Lie algebra $\R^n$ of
$\R^n / \Z^n$. Assume $(M_{\De}, \om_{\De})$ has a prequantum bundle
${\mathcal{L}}_{\De}$ and define the associated quantum spaces
$\mathcal{H}_k^\Delta = {\rm H}^0 (M_{\De} ,{\mathcal{L}}_{\De}^k)
$. Starting from the canonical basis $(e_i)$ of $\R^n$, we get $n$
operators
$$ T^{\De}_{i,k} :=T_{e_i, k} :\mathcal{H}_k^\Delta \rightarrow \mathcal{H}_k^\Delta, \qquad k \in \Z_{>0}, \; 
1 \leq i \leq n,$$
defined by Kostant-Souriau formula (\ref{eq:Txmux}).
\begin{cor} \label{cor:quantum-model} For any $k$, $T_{1,k}^{\De},
  \ldots , T_{n,k}^{\De}$ are mutually commuting operators with simple
  joint eigenspaces. Their joint spectrum is $(v+ \frac{2 \pi }{k}
  \Z^n) \cap \De$ where $v$ is any vertex of $\De$.
\end{cor}

\begin{proof} Apply Theorem \ref{theo:tt} to the polytope $\De - v$,
  whose vertices are integral.
\end{proof}

\section{Global quantum normal form}\label{sec:glob-quant-norm}

\subsection*{Toeplitz operators}

We briefly review Toeplitz operators.  Let $(M,\omega)$ be a compact
connected symplectic manifold with a prequantum line bundle
$\mathcal{L}$. Assume that $M$ is endowed with a complex stucture $j$
compatible with $\om$, so that $M$ is K{\"a}hler and $\mathcal{L}$ is
holomorphic. Here the holomorphic structure of the prequantum bundle
is the unique one compatible with the connection.  Recall that for a
positive integer $k$, $\mathcal{H}_k:=\mathrm{H}^0(M,\mathcal{L}^k)$
is the space of holomorphic sections of $\mathcal{L}^k$.

Since $M$ is compact, $\mathcal{H}_k$ is a closed finite dimensional
subspace of the Hilbert space ${\rm L}^2(M, \mathcal{L}^k)$. Here the
scalar product is defined by integrating the Hermitian pointwise
scalar product of sections agains the Liouville measure of $M$.
Denote by $\Pi_k$ the orthogonal projector of ${\rm L}^2(M,
\mathcal{L}^k)$ onto ${\mathcal{H}}_k$.

A \emph{Toeplitz operator} is any sequence $(T_k \colon \Hilbert_k
\rightarrow \Hilbert_k )_{k \in \mathbb{N}^*}$ of operators of the
form
\begin{eqnarray}
  \Big(T_k = \Pi_k f(\cdot,\, k) + R_k \Big)_{k\in \mathbb{N}^*} \label{qq}
\end{eqnarray}
where $f(\cdot,\,k)$, viewed as a multiplication operator, is a
sequence in $\Cinf( M)$ with an asymptotic expansion $f_0 + k^{-1} f_1
+ \ldots $ for the $\Cinf$ topology, and the norm of $R_k$ is
$\mathcal{O}(k^{-\infty})$.

We denote by $\mathscr{T} (M, {\mathcal{L}}, j ) $ the set of Toeplitz
operators. The following result corresponds to \cite[Theorem
1.2]{Ch2006a}.

\begin{theorem} \label{sc_algebra} The set $\mathscr{T}= \mathscr{T}
  (M, {\mathcal{L}}, j ) $ is a semi-classical algebra associated to
  $(M, \om)$ in the following sense. The set $\mathscr{T}$ is closed
  under the formation of product. So it is a star algebra, the
  identity is $(\Pi_k)_{k \in \mathbb{N}^*}$. The symbol map $\si_{\contravariant} :
  \Toeplitz \rightarrow \Cinf(M)[[\hb]],$ sending $T_k$ into the
  formal series $f_0 + \hb f_1 + ...$ where the functions $f_i$ are
  the coefficients of the asymptotic expansion of the multiplicator
  $f(\cdot,\, k)$, is well defined.  It is onto and its kernel is the
  ideal consisting of $\mathcal{O}(k^{-\infty})$ Toeplitz operators. More
  precisely for any integer $\ell$, $$ \| T_k \| =
  \mathcal{O}(k^{-\ell} ) \text{ if and only if } \si_{\contravariant}
  (T_k) = \mathcal{O}( \hb^{\ell}). $$ Furthermore, the induced product
  $*_{\contravariant}$ on $\Cinf (M) [[\hbar]]$ is a star-product.
\end{theorem}

We call the formal series $$\si _{\contravariant} (T_k)= f_0 + \hb f_1
+ \ldots $$ the \emph{contravariant symbol of} $(T_k)_{k \in \mathbb{N}^*}$. The first coefficient
$f_0$ is the {\em principal symbol} of $(T_k)_{k \in \mathbb{N}^*}$. The {\em subprincipal
  symbol} of $(T_k)_{\in \mathbb{N}^*}$ is the function
$$  g_1 = f_1 + \frac{1}{2} \Delta f_0,$$
where $\Delta$ is the holomorphic Laplacian of $M$. 

Consider two
Toeplitz operators with principal and subprincipal symbols $g_0, g_1$
and $g_0'$, $g_1'$ respectively. Then the principal and subprincipal
symbol of their composition is
$$ g_0'' + \hb g_1'' = (g_0 + \hb g_1)(g'_0 + \hb g'_1) + \frac{\hb }{2i} \{ g_0 , g_0' \} + O(\hb^2),$$ 
where $\{ \cdot, \cdot \}$ is the Poisson bracket of $(M, \om)$
(cf. Theorem 1.4 of \cite{Ch2006b}).

The Kostant-Souriau operators considered in Section
\ref{sec:quantum-model} are Toeplitz operators. More generally, let
$f$ be a function of $M$ with Hamiltonian vector field $X$. Applying
the Tuynman's trick (\cite{Tu1987}), one proves that the sequence
$$ T_k := \Pi_k \biggl( f + \frac{1}{{\rm i}k} \nabla_X  \biggr) , \qquad k \in \mathbb{N}^* $$
is a Toeplitz operator with principal symbol $f$ and subprincipal
symbol $- \frac{1}{2} \Delta f$.

\subsection*{Normal Form}

Recall that for each Delzant polytope $\De \subset \R^n$, we
introduced in Section \ref{sec:model-sympl-toric} a symplectic toric
manifold $(M_{\De},\, \om_{\De},\, \R^n/ \Z^n,\, \mu_{\De})$, a
complex structure $j_{\De}$ on $M_{\De}$ compatible with
$\om_{\De}$. Assume that $\De + c$ has integral vertices for some $c
\in \R^n$, so that $(M_{\De}, \om_{\De})$ has a prequantum bundle
${\mathcal{L}}_{\De}$ (unique up to isomorphism). We defined in
Section \ref{sec:quantum-model} for any positive $k$, commuting
operators $T_{1,k}^{\De}, \ldots, T_{n,k}^{\De}$ acting on the Hilbert
spaces $\mathcal{H}_k^\Delta = {\rm H}^0 (M_{\De} ,
{\mathcal{L}}^k_{\De}),\, k \in \Z _{>0}$, and described explicitly
their spectrum in Corollary \ref{cor:quantum-model}.

\begin{theorem}[Global normal form for a quantum toric system]
  \label{theo:normal-form}
  Let $(M, \, \omega, \,\R^n/ \Z^n, \, \mu )$ be a symplectic toric
  manifold equipped with a prequantum bundle ${\mathcal{L}}$ and a
  compatible complex structure $j$.  Denote by $\Delta$ the momentum
  polytope $\mu (M) \subset \R^n$. Let $T_1,\dots, T_n$ be commuting
  Toeplitz operators of $\mathscr{T} (M,{\mathcal{L}} , j)$ whose
  principal symbols are the components of $\mu$.

  Then there exists $k_0>0$, there exists a sequence
  $(g(\cdot;k))_{k\geq k_0}$ of smooth maps $\mathbb{R}^n \to
  \mathbb{R}^n$, and there exists an operator
  $U=(U_k:\mathcal{H}_k\to\mathcal{H}_k^\Delta)_{k\geq k_0}$ with
  $U_k$ invertible for any $k$, such that
$$
U_k(T_{1,k},\,\ldots,T_{n,k})U_k^{-1}=g(T_{1,k}^{\Delta},\, \ldots,
T^{\Delta}_{n,k}; k) + \mathcal{O}(k^{-\infty}).
$$
Moreover, $g$ admits an asymptotic expansion in the $\Cinf$ topology
of the form
\[
g(\cdot;k) =\textup{Id}+k^{-1}g_1+k^{-2}g_2+\cdots.
\]
If the operators $T_j$ are self-adjoint (i.e. for any $k$, $T_{j,k}$
is a self-adjoint operator), then $U_k$ may be chosen such that $U_k
\, U_k^*=\mathrm{Id}_{\mathcal{H}_k^\Delta}$.
\end{theorem}

\begin{remark}
  For small $k$, the dimensions of $\mathcal{H}_k$ and
  $\mathcal{H}_k^\Delta$ might be
  different. Theorem~\ref{theo:normal-form} does not give information
  about small values of $k$.
\end{remark}

The proof of Theorem~\ref{theo:normal-form} will require the following
technical lemma, which is a global version of a result of
Eliasson~\cite[Corollary page 14]{El1990}.  Recall that in the case
where $E$ is a closed half-space, $h\in\Cinf(\R^n_x\times E)$ if and
only if all the partial derivatives $\partial_x^k\partial_e^\ell
h(x,e)$ for $(x,e)\in \RM^n\times \mathring{E}$ have a limit at every
point in $\RM^n\times E$. (This is equivalent to saying that $g$ has a
smooth extension in a neighborhood of any point.)
\begin{lemma}\label{lemma:smooth}
  Let $E$ be a vector space or a closed half-space. Let
  $f\in\Cinf(\R^2_{(x,\xi)}\times E)$, and let
  $$q(x,\xi,e)=x^2+\xi^2.$$ Assume that $\{q,f\}=0$ (here the Poisson
  bracket refers to the symplectic variables $(x,\xi)$). Then there
  exists $g\in\Cinf(\RM_{\{\geq 0\}}\times E)$ such that
    $$f(x,\xi,e)
    = g(q(x,\xi,e),e).
    $$
  \end{lemma}

\begin{proof}
  Set-theoretically, there is a unique such function $g$. We need to
  prove that $g$ is smooth.  The Taylor expansion of $f$ in the
  $(x,\xi)$ variables has to commute with $q$. This implies that it
  has the form $$\sum_{k\geq 0} q^k a_k(e),$$ where
  $a_k\in\Cinf(E)$. Hence by the Taylor formula, for any integer
  $r\geq 0$, there is a polynomial $P_r$ in $q$ with coefficients in
  $\Cinf(E)$, and a smooth function $\phy\in\Cinf(\R^2\times E)$ such
  that
  \[
  f = P_r(q,e) + q^r \phy(x,\xi,e).
  \]
  Thus we get
  \[
  g(t,e) = P_r(t,e) + t^r \phy(\sqrt{t},0,e).
  \]
  When $t>0$, we simply compute the partial derivatives
  $\partial_t^k\partial_e^\ell g(t,e)$. They have a limit as
  $(t,e)\to(0,e)$ as long as $k\leq r$. Thus $g\in \op{C}^r(\R_{\geq
    0}\times E)$, which proves the lemma.
\end{proof}

\begin{proof}[Proof of Theorem \ref{theo:normal-form}]
  We divide the proof into several steps.

\paragraph{{\bf Step 1.}}
By Theorem \ref{theo:delzant}, there exists a
symplectomorphism $\phy:M\to M_\Delta$ such that
$\mu=\phy^*\mu_\Delta$. Since $M_\Delta$ is simply connected, the
prequantum bundle $\mathcal{L}_\Delta$ is unique up to
isomorphism. Hence $\phy$ can be lifted to a prequantum bundle
isomorphism $\mathcal{L}\to\mathcal{L}_\Delta$. So $\phy$ can be
quantized as an operator $U_k:\mathcal{H}_k\to \mathcal{H}_k^\Delta$
such that $U_k U_k^*=I_k$ for large $k$ and such that for any Toeplitz
operator $S=S_k$ with principal symbol $s$, $U_kS U_k^*$ is a Toeplitz
operator whose principal symbol is $s\circ\phy^{-1}$.

The operators $U_k$ that we use here have been introduced
in~\cite[Chapter 4]{Ch2003a} and similar ones have been considered
by~\cite{Zelditch1997}.  They are analogues of Fourier integral
operators \cite{Ho1971, DuHo1972}.

Replacing $T_j$ by $U T_j U^{*}$ we see that the problem is reduced to
the case where $T_1,\dots, T_n$ are commuting Toeplitz operators on
$(M_\Delta,\mathcal{L}_\Delta)$, with joint principal symbol equal to
$\mu_\Delta$.

\vspace{5mm}

\paragraph{{\bf Step 2.}}
We now prove the theorem by induction.  Assume that, for some
$N\in\NM$, we have
\begin{equation}
  (T_1,\dots,T_n) = g^{(N)}(T_1^\Delta,\dots,T_n^\Delta;k) + 
  k^{-(N+1)}R_{N+1},
  \label{equ:normal-form-induction}
\end{equation}
where $R_{N+1}$ is a vector of $n$ Toeplitz operators and $g^{(N)}$ is
polynomial in $k$:
\[
g^{(N)}=\textup{Id} + k^{-1}g_1+\cdots k^{-N} g_N,
\]
and each $g_j:\RM^n\to\RM^n$ is a smooth map. For simplicity we write
$$G_N:= g^{(N)}(T_1^\Delta,\dots,T_n^\Delta;k).$$

Notice that, for $N=0$, this is precisely the result of Step 1.  We
wish to prove that there exists an invertible Toeplitz operator
$U=(U_k)_{k\in\NM}$ such that
\begin{equation}
  U(T_1,\dots,T_n)U^{-1} = g^{(N)}(T^\Delta;k) +
  k^{-(N+1)}h_{N+1}(T^\Delta) + k^{-(N+2)}R_{N+2}.
  \label{equ:normal-form-induction-step}
\end{equation}
The procedure is standard and we only indicate the key points.  It
turns out that the case $N=0$ is slightly different from the other
cases $N>0$. When $N=0$, we plug~\eqref{equ:normal-form-induction} in
the left-hand side of~\eqref{equ:normal-form-induction-step} and
multiply on the right by $U$, and obtain
\[ [U,T^\Delta] + k^{-1}UR_1 = k^{-1}G_1U \mod k^{-2}\mathscr{T}.
\]
Since both sides of the equation are Toeplitz operators of order 1,
the equation is equivalent to the equality of the principal symbols~:
\[
\frac{1}i\{u,\mu_\Delta\} + ur_1 = g_1(\mu_\Delta)u.
\]
Writing $u$ of the form $u=\op{e}^{ia}$ we get the equation
\[
\{\mu_\Delta,a\} = r_1 - g_1(\mu_\Delta).
\]

For $N\geq 1$ we look for $U$ in the form $U_k=\textup{Id} +
{\rm i}k^{-N}A_N \mod k^{-(N+1)}\mathscr{T}$, where $A_N$ is a Toeplitz
operator. The same calculation as before gives the equation
\[
{\rm i}k^{-N} [A_N,G_N] + k^{-(N+1)}(R_{N+1} - K_{N+1}) = 0 \mod k^{-(N+2)}
\mathscr{T}.
\]
(We use here $2N+1\geq N+2$ in order to eliminate higher order terms.)
Since $$G_N=(T_1^\Delta,\dots T_n^\Delta)+\O(1),$$ the equation is
equivalent to the following equation on the principal symbols~:
\[
\{\mu_\Delta,a_N\} = r_{N+1} - h_{N+1}(\mu_\Delta).
\]

\vspace{5mm}

\paragraph{{\bf Step 3.}}
In order to complete the induction, we need to solve the following
cohomological equation, where the unknown functions are $a$ and
$g_j$, $1 \leq j \leq n$~:
\begin{equation}
  \{\mu_j^\Delta, a\}  = r_j - g_j\circ \mu_\Delta, \qquad 1 \leq j \leq n.
  \label{equ:cohomological}
\end{equation}
The proof follows Eliasson's local argument in~\cite[Lemma 8]{El1990},
where he uses a formula due to Moser. Here we show that this local
argument also works globally.

For any smooth function $r$ on $M_\Delta$, we define
\[
M_j r = \int_0^1 r\circ \phy_j^t \, {\rm d}t, \,\,\,\,\,\,\,\,\,\,\,\,\,
P_j r := \int_0^1 t r\circ \phy_j^t \, {\rm d}t.
\]
$M_j$ and $P_j$ are clearly linear operators sending $\Cinf(M_\Delta)$
into itself. Notice that, since the flows $\phy_j$ pairwise commute,
the Fubini formula ensures that $M_j$ and $P_k$ commute for any
$j,k$. The following Poisson bracket is easy to compute~:
\begin{align*}
  \{\mu_j^\Delta, P_j r\} &= \int_0^1t\{\mu_j^\Delta,r\circ\phy_j^t\}
  \, {\rm d}t= \int_0^1t\{\mu_j^\Delta\circ\phy_j^t,r\circ\phy_j^t\}
  \, {\rm d}t =
  \int_0^1t\{\mu_j^\Delta,r\}\circ\phy_j^t \, {\rm d}t\\
  & = \int_0^1t\frac{{\rm d}}{ {\rm d}t}(r\circ\phy_j^t) \, {\rm d}t =
  r - \int_0^1 r\circ\phy_j^t = r - M_j r.
\end{align*}

We shall need the following lemmas~.

\begin{lemma}
  \label{lemma:M}
  Let $r_1,\dots,r_n$ be smooth functions on $M_\Delta$ such that for
  all $i$, $j$,
  \[
  \qquad \{\mu_i^\Delta, r_j\} = \{\mu_j^\Delta, r_i\},
  \]
  then for all $1 \leq i,\,j \leq n$, we have that $\{M_jr_j, \mu_i^\Delta\} = 0$.
\end{lemma}
\begin{lemma}
  \label{lemma:g}
  Let $f\in\Cinf(M)$ such that for all $1 \leq i \leq n$ we have that
  $\{\mu_i^\Delta,f\}=0$. Then there exists $g\in\Cinf(\RM^n)$ such
  that
  $
  f = g \circ \mu_\Delta.
  $
\end{lemma}
\begin{proof}[Proof of Lemma~\ref{lemma:M}]
We have that
  \begin{align*}
    \{\mu_i^\Delta, M_jr_j \} &= \int_0^1 \{ \mu_i^\Delta,
    r_j\circ\phy_j^t\} \, {\rm d}t = \int_0^1 \{ \mu_i^\Delta,
    r_j\}\circ\phy_j^t \, {\rm d}t =\int_0^1 \{\mu_j^\Delta,
    r_i\}\phy_j^t \,
    {\rm d}t\\
    &= \int_0^1 \frac{{\rm d}}{{\rm d}t}(r_i\circ \phy_j^t) \, {\rm
      d}t= 0,
  \end{align*}
  as desired.
\end{proof}
\begin{proof}[Proof of Lemma~\ref{lemma:g}]
  Since $f$ is invariant by the action, set-theoretically, there
  exists unique function $g$ such that $f = g \circ \mu_\Delta$. We
  want to prove that $g$ is smooth.

  At a regular value of $\mu_\Delta$, this follows directly from the
  action-angle theorem.  Let $c$ be a critical value of $\mu_\Delta$,
  and let $C$ be a small ball around $c$. Let
  $(z_1,\dots,z_k,I_1,\theta_1,\dots, I_\ell,\theta_\ell)\in
  \C^k\times ({\rm T}^*S^1)^\ell$, with $k+\ell=n$ be Delzant coordinates on
  $(\mu_\Delta)^{-1}(C)$. Up to an affine transformation, we can
  assume
  $$
  \mu_\Delta=(\abs{z_1}^2/2,\dots,\abs{z_k}^2/2, I_1,\dots,I_\ell).
  $$

  By assumption, $f$ does not depend on the $\theta_j$ coordinates, so
  there is a smooth function $g_0$ such that
  $f=g_0(z_1,\dots,z_k,I_1,\dots,I_\ell)$.

  We apply Lemma~\ref{lemma:smooth} to the function $g_0$ with
  $(x,\xi)=z_1$ and $E=\C^{k-1}\times \RM^\ell$. Thus there is a
  smooth function $g_1$ such that
 $$f=g_1(\abs{z_1}^2,z_2,\dots,z_k,I_1,\dots, I_\ell).$$ We may now
  apply the same lemma to
  $$f_1(z,e)=g(e_1,z,e_2,\dots,e_k,e_{k+1},\dots, e_{n})$$ with
  $E=\RM_{\geq 0} \times \CM^{k-1}\times \RM^\ell$ and get a smooth
  function $g_2$ such that
$$
f=g_2(\abs{z_1}^2,\abs{z_2}^2,z_3,\dots,z_k,I_1,\dots, I_\ell).
$$
We may repeat the argument and finally obtain a smooth function $g_k$
such that
$$
f=g_k(\abs{z_1}^2,\dots,\abs{z_k}^2,I_1,\dots, I_\ell).
$$

This proves that $g$ is smooth in $C$. Thus $g$ is smooth on
$\mu_\Delta(M_\Delta)$ (which means that there is a smooth extension
of $g$ in $\RM^n$).
\end{proof}

We return now to the cohomological equation~\eqref{equ:cohomological}.
By Lemmas~\ref{lemma:M} and~\ref{lemma:g}, there exist smooth
functions $g_j$ on $\RM^n$ such that $M_j r_j= g_j\circ \mu_\Delta$.
Let
\[
a = P_1r_1 + P_2 M_1 r_2 + P_3 M_2 M_1 r_3 + \cdots P_n M_{n-1}\cdots
M_1 r_n.
\]
Notice that for any function $h$, $\{\mu_j^\Delta, M_j h\}=0$. Hence,
since the operators $M_j$ and $P_k$ commute, we get
\[
\{\mu_1^\Delta,a\} = \{\mu_1^\Delta,P_1 r_1\} = r_1 - M_1 r_1 =
g_1\circ \mu_\Delta,
\]
so $a$ solves the first equation of \eqref{equ:cohomological}.  Let
$$\tilde{a}=a-P_1r_1.$$ We have $\{\mu_1^\Delta,\tilde{a}\}=0$, and our
system becomes
\begin{equation}
  \{\mu_j^\Delta, \tilde{a}\}  = \tilde{r}_j - g_j\circ \mu_\Delta, \qquad 
  j=1,\dots,n.
  \label{equ:cohomological2}
\end{equation}
with
\begin{align*}
  \tilde{r}_j& :=r_j-\{\mu_j^\Delta,P_1r_1\} = r_j - \int_0^1
  t\{\mu_j^\Delta,r_1\circ\phy_1^t\}\, {\rm d}t = r_j - \int_0^1 t
  \{\mu_1^\Delta,r_j\}\circ \phy_1^t \, {\rm d}t\\
  &= r_j - \int_0^1 t \frac{{\rm d}}{{\rm d}t}(r_j\circ\phy_1^t)\,
  {\rm d}t = M_1 r_j.
\end{align*}
We notice that $$\tilde{a}= P_2 \tilde{r}_2 + P_3 M_2 \tilde{r}_3 +
\cdots P_n M_{n-1}\cdots M_2 \tilde{r}_n,$$ so by induction
$\tilde{a}$ solves the complete
system~\eqref{equ:cohomological2}. Thus $a$
solves~\eqref{equ:cohomological}.

The construction we have used to solve~\eqref{equ:cohomological} do
not require the functions $r_j$ to be real-valued. In case they are
real-valued, then $a$ and $g_j$ will be real-valued as well, and in
Step 2 we may choose $U_k=\exp(ik^{-N}A_N)$, which is unitary.

\vspace{5mm}

\paragraph{{\bf Step 4.}}

From steps 2 and 3 we obtain, for any positive integer $N$, an
invertible operator $U_N=(U_{N,k})_{\geq 0}$ (which is unitary in the
case of self-adjoint operators $T_j$) and a smooth map
$$g^{(N)}=\textup{Id} + k^{-1}g_1+\cdots k^{-N} g_N$$ such that
\begin{equation}
  U_N (T_1,\dots,T_n) U_N^{-1}= g^{(N)}(T_1^\Delta,\dots,T_n^\Delta;k)
  + k^{-(N+1)}R_{N+1},
  \label{equ:conjugationN} 
\end{equation}
where $R_{N+1}$ is a Toeplitz operator, and $U_N$ is of the form
\[
U_N = U^{(N)}U^{(N-1)}\cdots U^{(0)}.
\]

From step 3, we have $$U^{(j)}=\mathrm{Id}+ {\rm i}k^{-j}A_j \mod
k^{-j}\mathscr{T},$$ 
for $j\geq 1$. Therefore, one can construct by
induction a sequence of symbols $\tilde{a}_N$ such that for all $N\geq
1$, the operator $U^{(N)}U^{(N-1)}\cdots U^{(1)}$ is, modulo
$k^{-(N+1)}\mathscr{T}$, the Toeplitz quantization of the symbol
$$1+{\rm i}k^{-1}\tilde{a}_1 + \cdots + {\rm i}k^{-N}\tilde{a}_N.$$  By the Borel
summation procedure, one can find a Toeplitz operator $\tilde{A}$
whose total symbol has the asymptotic expansion
$$
\tilde{a}_1 + k^{-1}\tilde{a}_2+ \cdots + k^{-N+1}\tilde{a}_N+ \cdots.
$$
Moreover, one can find a smooth map $\tilde{g}$ that admits the
asymptotic expansion $$g^{(N)}=\textup{Id} + k^{-1}g_1+\cdots + k^{-N}
g_N+\cdots.$$ Now we let $\tilde{U}=(I+ik^{-1}\tilde{A})U^{(0)}$, so
that for any $N$, $$\tilde{U} = U_N \mod k^{-(N+1)}\mathscr{T}.$$ Thus
from~\eqref{equ:conjugation} we get, as required~:
\[
\tilde{U}(T_{1},\,\ldots,T_{n})\tilde{U}^{-1}=\tilde{g}(T_{1}^{\Delta},\,
\ldots, T^{\Delta}_{n}) + \mathcal{O}(k^{-\infty}).
\]

In the case where the operators $T_j$ are self-adjoint, one can change
the construction of the sequence $\tilde{a}_j$ in such a way that
$U^{(N)}U^{(N-1)}\cdots U^{(1)}$ is, modulo $k^{-(N+1)}\mathscr{T}$,
the \emph{exponential} of the Toeplitz quantization of the symbol
${\rm i}k^{-1}\tilde{a}_1 + \cdots + {\rm i}k^{-N}\tilde{a}_N$. Then we define
$$\tilde{U}:=\exp({\rm i}k^{-1}\tilde{A})U^{(0)},$$ which is unitary.
\end{proof}

\section{Isospectrality} \label{sec:inverse}

In this section we prove Theorem \ref{theo:spectral}
and Corollary \ref{theo:inverse-spectral}.

Recall that the \emph{joint
  spectrum of} of $n$ commuting matrix $A_1,\ldots,A_n$ is the set of
$(\lambda_1,\dots,\lambda_n)\in\CM^n$ such that there exists a
non-zero vector $v$ for which $A_j v = \lambda_j v$ for all
$j=1,\dots,n$. Such an $n$-uple $(\lambda_1,\dots,\lambda_n)$ will be
called a \emph{joint eigenvalue}.
We begin with the following elementary observations.

\begin{lemma} \label{lem:variational1}Let $B_1,\, \ldots,\,B_n$ be
  commuting self\--adjoint matrices.  Let $\epsilon>0$, and let $u \in
  \mathbb{R}^k \setminus \{0\}$ be such that $\|B_i \, u \|\leq
  \epsilon$ for all $1 \leq i \leq n$. Then there exists
  $\lambda\in\CM^n$ such that
$$
\lambda \in \op{JointSpec}(B_1,\, \ldots,\,B_n) \cap [-\epsilon,\,
\epsilon]^n.
$$
\end{lemma}
This lemma follows from the usual variational characterization of the
largest eigenvalue a matrix, applied to the self\--adjoint operator
$C:=(B_1^2+\cdots + B_n^2)^{\frac12}$.  It has the following immediate
consequences~:
\begin{lemma}\label{lemm:variational2}
The following statements hold.
  \begin{enumerate}[{\rm (i)}]
  \item If $B_1,\, \ldots,\,B_n$ are commuting self\--adjoint
    matrices, and $\alpha=(\alpha_1,\dots,\alpha_n)\in\CM^n$ is such
    that $$\|(B_i-\alpha_i)\,u\| \leq \epsilon \|u\|$$ for all $1 \leq i
    \leq n$, then there exists a joint eigenvalue $\lambda \in
    \op{JointSpec}(B_1,\, \ldots,\,B_n)$ such that $|
    \lambda_i-\alpha_i | \le \epsilon$ for all $1 \leq i \leq n$.
  \item Suppose that $A_1,\, \ldots,\,A_n$ is another collection of
    commuting self\--adjoint matrices, and assume
    \[ \|B_i-A_i\| \leq \epsilon \qquad \textup{for all}\,\, 1 \leq i \leq n,
    \]
    Then the Hausdorff distance between $\op{JointSpec}(A_1,\,
    \ldots,\, A_n)$ and $\op{JointSpec}(B_1,\, \ldots,\, B_n)$ is at
    most $\epsilon$, i.e.
    \[ {\rm d}_H\Big( \op{JointSpec}(A_1,\, \ldots,\, A_n), \,
    \op{JointSpec}(B_1,\, \ldots,\, B_n)\Big) \leq \epsilon.
    \]
  \end{enumerate}
\end{lemma}
\begin{proof}
  The first statement is obtained from Lemma~\ref{lem:variational1}
  applied to $B_j-\alpha_j I$. It implies that if $\alpha=(\alpha_1,\,
  \ldots,\, \alpha_n)$ is a joint eigenvalue of $(A_1,\, \ldots,\,
  A_n)$ and $\|B_i-A_i\| \leq \epsilon$ for all $1 \leq i \leq n$,
  then there exists a joint eigenvalue $\lambda \in
  \op{JointSpec}(B_1,\, \ldots,\, B_n)$ with $| \lambda_i-\alpha_i |
  \leq \epsilon$ for all $1 \leq i \leq n$ (and vice-versa), which
  gives the last statement.
\end{proof}

If $T_1,\, \ldots,\, T_n$ are pairwise commuting Toeplitz
operators, we call \emph{joint spectrum of $T_1,\ldots,T_n$} the
sequence of joint spectra of the set of commuting matrices
$(T_{1,k},\dots,T_{n,k})$ acting on the Hilbert space $\mathcal{H}_k$.

\begin{proof}[Proof of Theorem~\ref{theo:spectral}]
  By Theorem~\ref{theo:normal-form} applied to $(T_1,\dots,T_n)$, we
  get an integer $k_0>0$, a sequence $(g(\cdot;k))_{k\geq k_0}$ of
  smooth maps $\mathbb{R}^n \to \mathbb{R}^n$, and an operator
  $U=(U_k:\mathcal{H}_k\to\mathcal{H}_k^\Delta)_{k\geq k_0}$ with
  $U_k$ unitary, such that
  \begin{equation}
    U_k(T_{1,k},\,\ldots,T_{n,k})U_k^{-1}=g(T_{1,k}^{\Delta},\, \ldots,
    T^{\Delta}_{n,k}; k) + \mathcal{O}(k^{-\infty}).
    \label{equ:conjugation}
  \end{equation}
  Let $S_k:=\JointSpec(T_{1},\dots,T_n)$. If we introduce the
  components of $g$,
  \[
  g=:(g_1,\dots,g_n),
  \]
  then by Corollary \ref{cor:quantum-model}, the joint spectrum of the commuting Toeplitz
  operators
  \[
  \left(g_1(T_{1,k}^{\Delta},\, \ldots, T^{\Delta}_{n,k}; k), \dots,
    g_n(T_{1,k}^{\Delta},\, \ldots, T^{\Delta}_{n,k}; k)\right).
  \]
is $\Si_k := g(( v + \frac{2 \pi}{k} \ZM^n ) \cap \Delta, k )$.
 
Equation~\eqref{equ:conjugation} means that there exists a sequence
  $(C_N)_{N\in\NM}$ of real numbers such that
  \begin{equation*}
    \textup{for all}\,\, N, \,\,\,\, \textup{for all}\,\, k \qquad \norm{U_k
      (T_{1,k},\,\ldots,T_{n,k})U_k^{-1} - g(T_{1,k}^{\Delta},\, \ldots,
      T^{\Delta}_{n,k}; k)} \leq C_N k^{-N}.\label{equ:estimate}
  \end{equation*}
  Then, by virtue of Lemma \ref{lemm:variational2},
  we have 
  \begin{equation} \label{eq:estimate} \text{for all } N, \text{ for
      all } k, \qquad {\rm d}_H(S_k,\Sigma_k) \leq C_N k^{-N}
  \end{equation}
  which mean by definition
  $
  S_k= \Si_k + \mathcal{O}(k^{-\infty})$
  as we wanted to show.
\end{proof}

\begin{proof}[Proof of Corollary~\ref{theo:inverse-spectral}] 
Recall that we defined the limit of a sequence $ ( \mathcal{A}_k) _{k \in \mathbb{N}}$  of subsets of $\RM^n$ by 
$$
\lim \mathcal{A}_k:=\Big\{c \in \mathbb{R}^n \,\,\, |\,\, \, \forall U
\,\textup{neighborhood of}\, c, \exists k_0\, \textup{such that}\,
\forall k \ge k_0,\, U \cap \mathcal{A}_k \neq \emptyset \Big\},
$$
We denote by $\mathrm{B}(c,r)$ the open ball in $\RM^n$ centered at
$c$ and of radius $r$.  Notice that if $\mathcal{A}_k \subset
\mathcal{B}_k + \mathrm{B}(0,Ck^{-1})$, for some constant $C$, then
$\lim \mathcal{A}_k\subset \lim\mathcal{B}_k$.

We use the same notation as in the proof of Theorem \ref{theo:spectral}:
$
S_k=\JointSpec(T_{1},\dots,T_n)$ and 
$$
\Sigma_k = g(( v + \frac{2 \pi}{k} \ZM^n ) \cap \Delta, k).
$$

\paragraph{{\bf Step 1.}} The estimate~\eqref{equ:estimate} for $N=1$ gives
\[
S_k\subset \Sigma_k + \mathrm{B}(0,Ck^{-1}) \quad \text{ and
} \quad \Sigma_k\subset S_k + \mathrm{B}(0,C k^{-1}).
\]
Therefore, $\lim S_k = \lim \Sigma_k$.

\paragraph{{\bf Step 2.}}  Now let us show that $\lim \Sigma_k = \Delta$. From
Theorem~\ref{theo:normal-form}, we know that $g$ admits an asymptotic
expansion in the $\Cinf$ topology of the form
\[
g(\cdot;k) =\textup{Id}+k^{-1}g_1+k^{-2}g_2+\cdots.
\]
Therefore, for any compact $K\subset\RM^n$, there exists a constant
$C$ such that
\[
\max_{c\in K}\norm{g(c;k) - c}_{\RM^n}\leq Ck^{-1}.
\]
We may choose $K$ large enough so that it contains $\Delta$, and we
get the following inclusions~:
\begin{eqnarray} \label{eq:last} \Sigma_k\subset \Delta\cap \Bigl( v + \frac{2 \pi}{k}  \ZM^n \Bigr) + \mathrm{B}(0,Ck^{-1}) \quad \text{ and } \quad
  \Delta\cap \Bigl( v + \frac{2 \pi}{k} \ZM^n \Bigr) \subset \Sigma_k + \mathrm{B}(0,Ck^{-1}).
\end{eqnarray}
Hence $$\lim \Sigma_k = \lim{\Delta\cap ( v + \frac{2 \pi}{k} \ZM^n)}=\Delta.$$

From steps 1 and 2 we conclude that $\lim S_k = \Delta =
\mathrm{De}(M,\omega,\mu)$, which finishes the proof.
\end{proof}

\begin{cor}[Isospectrality]
  Two symplectic toric systems are isomorphic if and only if the limit of their joint spectra
  coincide. In particular, if two symplectic toric systems have the
  same joint spectra, then they are isomorphic.
\end{cor}

\section{Metaplectic correction}
\label{sec:half-form}

Theorem \ref{theo:spectral} and Theorem \ref{theo:inverse-spectral}
can be proven analogously to our proofs in the presence of a half form
bundle. Below we explain the corresponding modifications needed.

Toeplitz quantization is often considered more natural in presence of
a half-form bundle. Let $(M, \om , j )$ be a compact K{\"a}hler
manifold. A half-form bundle of $(M,j)$ is a square root of the
canonical bundle of $M$. More precisely we consider a pair $(\delta,
\varphi)$ consisting of a complex line bundle $\delta \rightarrow M$
and an isomorphism $$\varphi : \delta^{\otimes 2} \rightarrow
\wedge^{n,0} {\rm T}^*M.$$ Here $n$ is the complex dimension of $M$. Such a
square root does not necessarily exists, and if it exists, the space
of half-form bundles up to isomorphism is a principal homogeneous
space for the group ${\rm H}^1 ( M, \Z / 2 \Z)$.

Let $\mathcal{L} \rightarrow M$ be a prequantum bundle and $(\delta,
\varphi)$ be a half-form bundle. Observe that $\delta$ has a metric
and a holomorphic structure determined by the condition that $\varphi$
is an isomorphism of Hermitian holomorphic bundles. Define the quantum
space $\Hilb_{\op{m},k}$ as the vector space of holomorphic sections of $\mathcal{L}^k
\otimes \delta$. The space $\Hilb_{\op{m},k}$ has a natural scalar product obtained by
integrating the point wise scalar product of sections of $\mathcal{L}^k \otimes
\delta$ against the Liouville measure. This scalar product is actually
defined on the space of ${\rm L}^2$ sections, and we have an orthogonal
projector $\Pi_k$ from the space of ${\rm L}^2$ sections onto $\Hilb_{\op{m},k}$. The
definition of Toeplitz operators is the same as before except that we
use this new projector. So a Toeplitz operator is any family $(T_k
\colon \Hilbert_{\op{m},k} \rightarrow \Hilbert_{\op{m},k} )_{k\in \mathbb{N}^*}$ of
endomorphisms of the form
\begin{eqnarray*}
  T_k = \Pi_k f(\cdot,\, k) + R_k , \qquad k\in \mathbb{N}^*
\end{eqnarray*}
where $f(\cdot,\,k)$, viewed as a multiplication operator, is a
sequence in $\Cinf( M)$ with an asymptotic expansion 
$$f_0 + k^{-1} f_1
+ \ldots$$ for the $\Cinf$ topology, and the norm of $R_k$ is
$\mathcal{O}(k^{-\infty})$. Theorem \ref{sc_algebra} still holds and
we define the principal and subprincipal symbols of a Toeplitz
operators with the same formula as before. The rule of composition of
these symbols is still given by (\ref{eq:comp_symbol}). We can also
define Toeplitz operators by using the Kostant-Souriau
formula. Consider a smooth function $f$ on $M$ whose Hamiltonian
vector field $X$ preserves the complex structure. Then the following
operators are well defined
\begin{gather} \label{eq:comp_symbol} T_k = f + \frac{1}{{\rm i}k} (
  \nabla_{X}^{\mathcal{L}^k} \otimes {\rm Id} + {\rm Id} \otimes \op{L}^\delta_X ) :
  \Hilb_{\op{m},k} \rightarrow \Hilb_{\op{m},k}, \qquad k \in \mathbb{N}^*.
\end{gather}
Here $\op{L}^\delta$ is the Lie derivative of half-form, that is the
first order differential operator such that for any local section $s$
of the half-form bundle one has $$ \op{L}_X ( \varphi ( s^{\otimes 2}))
= 2 \varphi ( s \otimes \op{L}^\delta_X s ).$$ One also shows that the
family $( T_k)$ is a Toeplitz operator with principal symbol $f$ and
vanishing subprincipal symbol.

Consider now the Delzant space $M_{\Delta}$ defined as in Section
\ref{sec:model-sympl-toric}. Since $M_\Delta$ is simply connected,
there exists at most one half-form bundle.

Recall that we denote by $F$ the set of faces of codimension 1 and for
each $f \in F$, $X_f \in \frak{t} $ is the primitive normal vector to
the face.
\begin{prop} \label{prop:existence-half-form} The Delzant space
  $M_{\Delta}$ admits an equivariant half-form bundle if and only if
  there exists $\ga \in \frak{t}_{\Z}^*$ such that $ \ga ( X_f)$ is
  odd for any $f \in F$.
\end{prop}

\begin{proof} This criterion may be established using the divisor of
  the toric variety, cf. \cite{CLS}. To each face $f \in F$
  corresponds an irreducible divisor $D_f$ of $M_\De$. It is known
  that a divisor of the canonical bundle is $- \sum_{f \in F}
  D_f$. Recall also that the divisors $D_f$'s generate the Picard
  group, and that $\sum n_f X_f$ is principal if and only if $n_f =
  \ga ( X_f)$ for some $\ga \in \frak{t}^*_{\Z}$. So $M_{\Delta}$
  admits a half-form bundle if and only if there exists a divisor $$D=
  \sum n^\delta_f D_{f}$$ such that $ 2D + \sum_{f \in F} D_f$ is
  principal, that is
$$ 2 n^\delta_f + 1 =   \langle X_f, \ga \rangle $$ 
for some $\ga \in \frak{t}_{\Z}^*$.
\end{proof}

Assume now that the vertices of $\Delta$ belongs to $2\pi
\got{t}_{\Z}^*$ so that $( M_{\Delta} , \om_{M_{\Delta}})$ admits a
prequantum bundle ${\mathcal{L}}_{\Delta}$ unique up to
isomorphism. Assume also that $M_\Delta$ is equipped with a half-form
bundle $\delta_{\De}$. For any positive integer $k$, define the
quantum space
$$ \Hilb_{\op{m}, k}^{ \Delta} = {\rm H}^0 ( M_{\Delta} , {\mathcal{L}}^k_{\Delta} \otimes \delta_{\De} ).$$  
For any $X \in \mathfrak{t}$, consider the rescaled Kostant-Souriau
operators
\begin{eqnarray}  T_{X,k} := \langle \mu_\Delta , X
  \rangle + \frac{1}{{\rm i}k} \Bigl( \nabla_{X^{\sharp}} \otimes {\rm Id} + {\rm Id}
  \otimes \op{L}^{\delta}_X \Bigr) : \Hilb_{\op{m}, k}^{\Delta}
  \rightarrow \Hilb_{\op{m}, k}^{\Delta}
\end{eqnarray}
We can now state the analogue of Theorem \ref{theo:tt}. Introduce $\ga
\in\frak{t}_{\Z}^*$ such that $\ga ( X_f) $ is odd for any $f\in F$.

\begin{theorem} \label{theo:ttm} There is an orthogonal decomposition of
the quantum space $ \Hilb_{\op{m}, k}^{\Delta}$
  into a direct sum of lines:
$$ \Hilb_{\op{m}, k}^{\Delta}  = \bigoplus_{\ell \in (\frac{2\pi}{k}( \mathfrak{t}_{\Z}^* + \frac{1}{2} \ga )) \cap \Delta}  D_{\ell}^k$$
such that, for any $X \in {\mathfrak{t}}$,
$$ T_{X,k} \Psi = \ell (X) \Psi , \qquad \textup{for all}\,\,\, \Psi \in D_{\ell}^k.$$ 
\end{theorem}

Observe that the points of $(\frac{2\pi}{k}( \mathfrak{t}_{\Z}^* +
\frac{1}{2} \ga )) \cap \Delta$ are all in the interior of
$\Delta$. Furthermore on a neighborhood of each vertex we recover the
usual joint spectrum of $n$ harmonic oscillators as follows. For any
vertex $v$ denote by $F_v$ the set of one-codimensional faces adjacent
to $v$ so that $(f_v, \; v \in F_v)$ is a basis of $t_{\Z}^*$. Then
there exists a neighborhood $U$ of $v$ such that
$$\Delta \cap U = \{ v + x / \langle f, x \rangle \geqslant 0, \forall f \in F_v \} \cap U $$ and 
$$ (\tfrac{2\pi}{k}( \mathfrak{t}_{\Z}^* + \tfrac{1}{2} \ga )) \cap \Delta \cap U = \{ v + x / \langle f, x \rangle \in \tfrac{2 \pi}{k} ( \N + \tfrac{1}{2} ) , \forall f \in F_v \} \cap U .$$

\begin{proof}
  Let us adapt the proof of Theorem \ref{theo:tt}.  First we introduce
  the quantization of $\C^F$ with metaplectic correction.  Choose an
  ordering of $F$ and define $v \in \Om ( \C^F) $ as the wedge product
  of the $\op{d}\!z_f$ 's. $v$ is a non vanishing section of the
  canonical bundle of $\C^F$. The action of $t \in {\mathbb{R}}^F /
  {\mathbb{Z}}^F$ on the canonical bundle sends $ v$ into $ \exp ( -
  2{\rm i} \pi \sum t_f) v$.

  Let $\delta_F $ be the trivial complex line bundle with base $\C^F$
  and $\varphi_F $ be the isomorphism from $\delta_F^2$ to the
  canonical bundle of $\mathbb{C}^F$ given by $\varphi(z,1)= v(z)
  $. By Proposition \ref{prop:existence-half-form}, there exist $d \in
  \Z^F$ and $\ga \in \mathfrak{t}_{\Z}^*$ such that $$\pi^* \ga =
  \mathbb{I} + 2 d$$ where $\mathbb{I}$ is the vector in
  $(\mathbb{R}^F)^*$ with all components equal to $1$. Consider the
  action of ${\mathbb{R}}^F / {\mathbb{Z}}^F$ on $\delta_F$ given by
$$ t \cdot  (z, u) = (t \cdot z , e^{ 2{\rm i} \pi \sum t_f d_f } u)$$
Then the isomorphism $\varphi$ intertwines the action of the subtorus
$N$ of ${\mathbb{R}}^F / {\mathbb{Z}}^F$ on $\delta_F$ with the action
of $N$ on the canonical bundle. This condition has the consequence
that the quotient of $\delta_F^2$ by $N$ defines a half-form bundle
$\delta_\Delta$ on the Delzant space $M_{\Delta}$, cf. section 8.3 of
\cite{Ch2010}. Furthermore one has an isomorphism
$$ \Phi_k : (\mathcal{B}_{\op{m}, k} )^N  \rightarrow  \Hilb_{\op{m}, k}^{ \Delta}$$
from the $N$-invariant part of the space $\mathcal{B}_{\op{m}, k}$ of
holomorphic sections of $\mathcal{L}^k_F \otimes \delta_F$ to the
quantum space $\Hilb_{\op{m}, k}^{ \Delta}$.

Introduce the rescaled Kostant-Souriau operators
\[
S_{X,k}: = \pscal{\tilde{\mu}}{X} + \frac{1}{{\rm i}k}
(\nabla^{\mathcal{L}^k}_{X^{\sharp}} \otimes \op{Id}+ \op{Id} \otimes
\op{L}^{\delta}_{X^{\sharp}}).
\]
Then $ (\mathcal{B}_{\op{m}, k} )^N$ consists of the sections $\Psi
\in \mathcal{B}_{\op{m}, k}$ such that $S_{X,k} \Psi = 0 $ for any $X
\in \mathfrak{n}$. Furthermore
$$ \Phi_k \Bigl(  S_{X,k} \Psi \Bigr) = T_{\pi(X),k} \Phi_k ( \Psi), \qquad \textup{for all}\,\,\,\, \Psi \in (\mathcal{B}_{\op{m}, k})^N  .$$ 
To conclude the proof let us compute the action of the rescaled
Kostant-Souriau operators. We have for $X=e_f$
$$S_{X,k} \bigl( {\rm e}^{-\frac{k}{8 \pi }\abs{z}^2} g(z)\bigr) =
{\rm e}^{-\frac{k}{8 \pi }\abs{z}^2} \Bigl( \frac{ 2 \pi}{k} \bigl(
z_f\partial_{z_f}g + \tfrac{1}{2} g \bigr) - \lambda_f g \Bigr).$$ For
any $\alpha \in \mathbb{N}^F$ set $ \psi_\alpha =
e^{-\frac{k}{2}\abs{z}^2}z^\alpha$ so that
\begin{eqnarray} \label{eq:tx2} S_{X,k}(\psi_\alpha)
  =\pscal{X}{\frac{2 \pi }{k} (\al + \tfrac{1}{2} \mathbb{I} )
    -\lambda } \psi_\alpha.
\end{eqnarray}
Hence the space $(\mathcal{B}_{\op{m}, k} )^N$ admits as a basis the
family $( \Psi_{\al})$ where $\al$ runs over $ ( \Z^F + \frac{1}{2}
\mathbb{I}) \cap \ker \iota_{\mathfrak{n}}^*$. To conclude we prove as
in Lemma \ref{lemm:bij} that $\pi ^* + \la$ restricts to a bijection
$$\frac{ 2 \pi }{k} ( \mathfrak{t}_{\Z}^* + \frac{1}{2} \ga) \longrightarrow 
\frac{ 2 \pi }{k} ( \Z^F + \frac{1}{2} \mathbb{I}) \cap \ker
\iota_{\mathfrak n } ^*. $$
\end{proof}

 \begin{figure}[h]
   \centering
   \includegraphics[height=2.5cm]{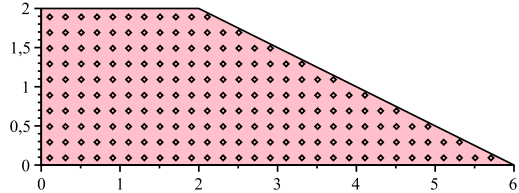}
   
   \caption{Spectra in Figure \ref{fig:deformed} with metaplectic correction.}
   \label{fig:deformed_metaplectic}
 \end{figure}

Let us consider now the generalization of Theorem \ref{theo:normal-form}. 
 Let $(M, \, \omega, \,\R^n/ \Z^n, \, \mu )$ be a symplectic toric
  manifold equipped with a prequantum bundle ${\mathcal{L}}$, a
  compatible complex structure $j$ and a half-form bundle $\delta$.
 Denote by $\Delta$ the momentum polytope $\mu (M) \subset \R^n$. Let $T_1,\dots, T_n$ be commuting self-adjoint Toeplitz operators of $\Hilb_{\op{m},k}$  whose
  principal symbols are the components of $\mu$. Denote by $f^i_1$ the subprincipal symbol of $T_i$. 

\begin{theorem}
  There exists $k_0>0$, there exists a sequence
  $(g(\cdot;k))_{k\geq k_0}$ of smooth maps $\mathbb{R}^n \to
  \mathbb{R}^n$, and there exists an operator
  $U=(U_k:\mathcal{H}_k\to\mathcal{H}_k^\Delta)_{k\geq k_0}$ with
  $U_k$ unitary for any $k$, such that
$$
U_k(T_{1,k},\,\ldots,T_{n,k})U_k^{-1}=g(T_{1,k}^{\Delta},\, \ldots,
T^{\Delta}_{n,k}; k) + \mathcal{O}(k^{-\infty}).
$$
Moreover, $g$ admits an asymptotic expansion in the $\Cinf$ topology
of the form $\textup{Id}+k^{-1}g_1+k^{-2}g_2+\cdots$ 
where the subprincipal term $g_1$ is given by 
$$ g^i_1 (E) = \int^{0}_{1} f^i_1 (\varphi^t_{\mu_i} (x)) \, {\rm d}t , \qquad \textup{for all}\,\,\, \, E \in \Delta, \, x \in \mu^{-1}(E) .$$
Here $\varphi_{\mu_i}^t$ is the Hamiltonian flow of $\mu_i$. 
\end{theorem}

\begin{proof} 
  The proof is the same as the one of Theorem \ref{theo:normal-form}. Because of the metaplectic correction, we can choose in the first step the operator $U_k$ quantizing $\varphi$ in such a way that for any Toeplitz operator $(S_k)$, $(S_k)$ and $(U_k S_k U_k^*)$ have the same principal and subprincipal symbols. This follows from Theorem 5.1 in \cite{Ch2007}. From this we can extract $g_1$ from the cohomological equation $\{ \mu_i, a \}= f^i_1 - g_1^i(\mu_\Delta) $.  
\end{proof}

Finally we deduce Theorem \ref{theo:metapl-corr} in the introduction by following the same method as in Section \ref{sec:inverse}.

\section{Final Remarks} \label{sec:remarks}

 In the present
paper we have dealt with isospectrality in the context of integrable
systems and symplectic geometry. The paper settles the \emph{Spectral Goal
  for Quantum Systems} for the case of toric systems outlined in the
last two authors' article \cite{PeVN2013}: to prove that large classes
of integrable systems are determined by their semiclassical joint
spectrum. 

This type of inverse question fits in the framework of
``isospectral questions": what is the relation between two operators
that have the same spectrum? The question of isospectrality has been
considered by many authors in different contexts, and may be traced back
to a more general question of by H. Weyl \cite{We1911,We1912}. 
The question is perhaps most
famous thanks to Kac's article \cite{Ka66} (who attributes the question to
S. Bochner), which also popularized the
phrase: ``can one hear the shape of a drum?".

\subsection*{Isospectrality in geometry} \label{sec:remarks1}

Corollary
  \ref{theo:inverse-spectral} says that the joint spectrum does indeed
  determine the system. This type of conclusion often has a negative
  answer, at least if one considers it in Riemannian geometry. In
  Riemannian geometry the operator whose spectrum is considered is the
  Laplace operator.  Two compact Riemannian manifolds are said to be
  \emph{isospectral} if the associated Laplace operators have the same
  spectrum.

 Bochner and Kac's question has a negative answer in this case, even for planar
  domains with Dirichlet boundary conditions (which is the original
  version posed in \cite{Ka66}).    There are many works in this
  direction, see for instance Milnor \cite{Mi1964}, Sunada \cite{Su85}, Berard \cite{Be92},
  and Buser \cite{Bu86} and Gordon\--Webb\--Wolpert \cite{GoWeWo92,
    GoWeWo92b}. As we have mentioned, in symplectic geometry a few
  positive results are known.  These results, and the present paper,
  give evidence that symplectic invariants seem to be more encodable
  in the spectrum of a quantum integrable system than Riemannian
  invariants in the spectrum of the Lapace operator.

Inverse type results in the
realm of spectral geometry have
been obtained by many other authors, see for instance
 Br{\"u}ning\--Heintze \cite{Br1984b},   Colin de Verdi{\`e}re \cite{CdV, CdV2, CdV3}, 
  Colin de Verdi{\`e}re\--Guillemin \cite{CdV4}, 
  Croke\--Sharafutdinov \cite{CrSh1998},
  Guillemin\--Kazhdan \cite{GuKa1980},   
  McKean\--Singer \cite{McSi1967},  
   Osgood\--Phillips\--Sarnak \cite{OsPhSa1989}, and Zelditch \cite{Ze2009},
   and the references therein.  
  An interesting general problem (for instance in the context of toric geometry) is to what extent information about
measures may be recovered from the spectrum, see Guillemin\--Sternberg \cite[p. 72\--78]{GuSt1984}
for a result in this direction.

\subsection*{Isospectral conjecture for semitoric systems}

In \cite{PeVN2010, PeVN2012} the last two authors formulated an
inverse spectral conjecture for semitoric completely integrable
systems (see \cite{PeVN2009,PeVN2011} for a classification of
semitoric systems in terms of five symplectic invariants): the
semiclassical joint spectrum of a quantum semitoric system determines
the corresponding classical system.

 \begin{figure}[h]
   \centering
   \includegraphics[width=0.4\linewidth]{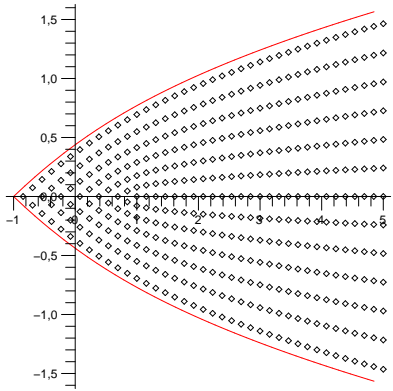}
   \caption{The quantum coupled
     spin\--oscillator is a fundamental example of quantum semitoric integrable system. 
     Its joint spectrum is depicted in the figure for a fixed value of the spectral parameter.}
   \label{fig:spectrumapprox2versionsmall}
 \end{figure}

\emph{Semitoric systems} are four-dimensional integrable systems with
two degrees of freedom for which one component of the system generates
a $2\pi$-periodic flow.  Semitoric systems lie somewhere in between
toric systems and general integrable systems. If both components of
the semitoric system are $2\pi$-periodic, i.e. the system is generated
by a Hamiltonian $2$-torus action, then the system is a \emph{toric
  system} (strictly speaking after a harmless rescaling of the
periods).

Theorem~\ref{theo:inverse-spectral} above solves the conjecture in the
class of toric systems. In this class the result is even stronger,
since there is no restriction on the dimension, and moreover only the
spectrum modulo $\mathcal{O}(\hbar)$ is needed, whereas in general one
expects that an accuracy of order $\mathcal{O}(\hbar^2)$ is necessary.
\\
\\
\\
{\bf Acknowledgements.}  We thank Jochen Br{\"u}ning, 
Helmut Hofer, Peter Sarnak, and
Thomas Spencer for fruitful discussions.  
The authors are grateful to Helmut Hofer for his essential support, which
made it possible for LC and VNS to visit AP at the Institute for
Advanced Study during the Winter and Summer of 2011, where a part of
this paper was written. Additional financial support for the visits
was provided by Washington University and NSF.

AP was partly supported by an NSF Postdoctoral Fellowship, a MSRI
membership, an IAS membership, NSF Grants DMS-0965738 and DMS-0635607, an NSF CAREER
Award, a Leibniz Fellowship from the Mathematisches Forschungsinstitut
Oberwolfach, Spanish Ministry of Science Grant MTM 2010-21186-C02-01,
and by the Spanish National Research Council (CSIC).  
VNS was partly supported by the NONAa grant from the French ANR and 
the Institut Universitaire de France.

\bibliographystyle{new}

\medskip\noindent
Laurent Charles\\
Institut de Math{\'e}matiques de Jussieu\\
Universit{\'e} Pierre et Marie Curie (Paris VI), Case 247\\
4, place Jussieu \\
F-75252 PARIS CEDEX 05.

\noindent
\\
{\'A}lvaro Pelayo \\
School of Mathematics\\
Institute for Advanced Study\\
Einstein Drive\\
Princeton, NJ 08540 USA.
\\
\\
and
\\
\\
\noindent
Washington University,  Mathematics Department \\
One Brookings Drive, Campus Box 1146\\
St Louis, MO 63130-4899, USA.\\
{\em E\--mail}: \texttt{apelayo@math.wustl.edu}

\medskip\noindent

\noindent
\noindent
San V\~u Ng\d oc \\
Institut Universitaire de France
\\
\\
Institut de Recherches Math\'ematiques de Rennes\\
Universit\'e de Rennes 1, Campus de Beaulieu\\
F-35042 Rennes cedex, France\\
{\em E-mail:} \texttt{san.vu-ngoc@univ-rennes1.fr}

\end{document}